\documentclass[10pt]{amsart}
\usepackage{amsxtra, amsfonts, amsmath, amsthm, amstext, amssymb, amscd, mathrsfs, verbatim, color}
\usepackage{threeparttable}
\usepackage{mathtools}
\usepackage[ansinew]{inputenc}\usepackage[T1]{fontenc}
\usepackage[all,cmtip]{xy}
\theoremstyle{plain}

\newtheorem{theorem}{Theorem}[section]
\newtheorem{lemma}[theorem]{Lemma}
\newtheorem{definition-theorem}[theorem]{Definition-Theorem}
\newtheorem{proposition}[theorem]{Proposition}
\newtheorem{corollary}[theorem]{Corollary}

\theoremstyle{definition}
\newtheorem{definition}[theorem]{Definition}
\newtheorem{example}[theorem]{Example}
\newtheorem{remark}[theorem]{Remark}
\newtheorem{notation}[theorem]{Notation}
\newcommand \bth[1] { \begin{theorem}\label{t#1} }
\newcommand \ble[1] { \begin{lemma}\label{l#1} }

\newcommand \bpr[1] { \begin{proposition}\label{p#1} }
\newcommand \bco[1] { \begin{corollary}\label{c#1} }
\newcommand \bde[1] { \begin{definition}\label{d#1}\rm }
\newcommand \bex[1] { \begin{example}\label{e#1}\rm }
\newcommand \bre[1] { \begin{remark}\label{r#1}\rm }

\newcommand \bnota[1] {\begin{notation}\label{n#1}\rm }
\newcommand {\ele} { \end{lemma} }

\newcommand {\epr} { \end{proposition} }
\newcommand {\eco} { \end{corollary} }
\newcommand {\ede} { \end{definition} }
\newcommand {\eex} { \end{example} }
\newcommand {\ere} { \end{remark} }
\newcommand {\enota} { \end{notation} }

\def \Id { {\mathrm{Id}} }

\DeclareMathOperator \tr { {\mathrm{tr}} }

\DeclareMathOperator \Hom { {\mathrm{Hom}} }
\DeclareMathOperator \Ext { { \mathrm{Ext}}}

\DeclareMathOperator \Ind { {\mathrm{Ind}} }

\DeclareMathOperator \Res { {\mathrm{Res}} }

\DeclareMathOperator \sgn { { \mathrm{sgn}}}

\DeclareMathOperator \im { { {\mathrm im}}}

\begin{document}
\setlength{\baselineskip}{1.2\baselineskip}
\title[Duality for Ext-groups]
{Duality for Ext-groups and extensions of discrete series for graded Hecke algebras}
\author[Kei Yuen Chan]{Kei Yuen Chan}
\address{ Korteweg-de Vries Institute for Mathematics, Universiteit van Amsterdam}
\email{K.Y.Chan@uva.nl}

\maketitle 
\begin{abstract}
In this paper, we study extensions of graded affine Hecke algebra modules. In particular, based on an explicit projective resolution on graded affine Hecke algebra modules, we prove a duality result for $\mathrm{Ext}$-groups. This duality result with an $\mathrm{Ind}$-$\mathrm{Res}$ resolution gives an algebraic proof of the fact that all higher $\mathrm{Ext}$-groups between a tempered module and a discrete series vanish. 
\end{abstract}

\section{Introduction} \label{s intro}

\subsection{}
Graded affine Hecke algebras were defined by Lusztig \cite{Lu} for the study of the representations of affine Hecke algebras and $p$-adic groups. The relation between affine Hecke algebras and their graded ones can be thought as an analog of the relation between Lie groups and Lie algebras, and so graded affine Hecke algebras are simpler in certain aspects. 

The classification of irreducible graded Hecke algebra modules has been studied extensively. A notable result is the Kazhdan-Lusztig geometric classification \cite{KL} for equal parameter cases. In arbitrary parameters, general results include the Langlands classification \cite{Ev} and a classification of discrete series by Opdam-Solleveld \cite{OS1}. There are also some other classification results \cite{CK, Ka, Kri, KR, Sl}.


This paper is to study another aspect of graded affine Hecke algebra modules, that is extensions between modules. The extension problem is a natural question after thorough understanding on the classification of irreducible graded Hecke algebra modules. Our goal in this paper is to establish some general results for extensions, and we do not use any explicit classification of simple modules in this paper. 

There are some related studies of the extension problem in the literature \cite{AP, BW, Ka2, Me, OS, OS2, Or,Pr, SS, So, Vi}. While our work is motivated from some known results in the setting of $p$-adic groups and affine Hecke algebras, our approach is self-contained in the theory of graded affine Hecke algebras. Moreover, because of the algebraic nature of our approach, results apply to the graded Hecke algebra of non-crystallographic types (see \cite{Kri} and \cite{KR}) and complex parameters, which seems to be new in the literature. This extends some results such as the ones in \cite{CT} to a general graded Hecke algebra. It may be interesting to consider our approach in some similar algebraic structures such as degenerate affine Hecke-Clifford algebras \cite{Na} and graded Hecke algebras for complex reflection groups \cite{SW}. 

\subsection{}
There are two main results in this paper. The first one establishes a duality of $\mathrm{Ext}$-groups. The second one is to apply the duality result with a resolution arisen from $\mathrm{Ind}$-$\mathrm{Res}$ functor to compute extensions between a tempered module and a discrete series. 

The duality result involves two main ingredients. Let $\mathbb{H}$ be a graded affine Hecke algebra (see Definition \ref{def graded affine}). The first ingredient is an explicit construction of a projective resolution on $\mathbb{H}$-modules, which makes computations possible. The projective resolution is an analogue of the classical Koszul resolution for relative Lie algebra cohomology. The second main ingredient is three operations $*$, $\bullet$, $\iota$ on $\mathbb{H}$ (see Section \ref{sec duals} and Section \ref{ss IM invo} for the detailed definitions).

Those three operations play some roles in the literature. The first anti-involution $*$ arises naturally from the study of unitary duals for the Hecke algebra of a $p$-adic group (see \cite{BM, BM1}). We remark that our definition indeed differs from the one in \cite{BM} by a complex conjugate, which is more convenient for our study. The $*$-dual (in our notation) of an $\mathbb{H}$-module indeed corresponds to the contragredient representation in the $p$-adic group level (see \cite[Section 5]{BM1}). The second anti-involution $\bullet$ is studied in a recent paper of Barbasch-Ciubotaru \cite{BC} as an Hecke algebra analogue of the  compact-star operation for $(\mathfrak{g}, K)$-modules in \cite{ALTV, Ye}. The $\bullet$-operation is also studied by Opdam \cite{Op0} in the Macdonald theory for affine Hecke algebras. This $\bullet$-operation naturally appears in the content of Iwahori-Hecke algebras via the projective (non-admissible) modules defined by Bernstein (see \cite{BC2}). The last operation $\iota$ on $\mathbb{H}$ is the Iwahori-Masumoto involution, which is shown by Evens-Mirkovi\'c \cite{EM} to have close connection with the geometric Fourier-Deligne transform.

For each of the operations $*, \bullet, \iota$, it induces a map from the set of $\mathbb{H}$-modules to the set of $\mathbb{H}$-modules. For an $\mathbb{H}$-module $X$, we denote by $X^*$, $X^{\bullet}$ and $\iota(X)$  (see Section \ref{sec duals} and Section \ref{ss IM invo}) for the corresponding dual $\mathbb{H}$-modules respectively. We also define $\mathbb{D}(X)=\iota(X^{\bullet})^*$. 

Our first main theorem is the following duality on the $\mathrm{Ext}$-groups:

\begin{theorem} \label{thm poin dua intro} (Theorem \ref{thm poin dua}, Theorem \ref{thm dua YP}) Let $\mathbb{H}$ be the graded affine Hecke algebra associated to a root datum $( R, V, R^{\vee}, V^{\vee}, \Pi)$ and a parameter function $k: \Pi\rightarrow \mathbb{C}$ (Definition \ref{def graded affine}). Let $n=\dim V$.  Let $X$ and $Y$ be finite dimensional $\mathbb{H}$-modules. Then we have the following:
\begin{enumerate}
\item[(1)] there exists a natural non-degenerate pairing 
\[  \mathrm{Ext}^i_{\mathbb{H}}(X, Y) \times \mathrm{Ext}_{\mathbb{H}}^{n-i}(X^*, \iota(Y)^{\bullet})  \rightarrow \mathbb{C}.
\]
\item[(2)] the Yoneda product
\[ \mathscr{Y}_i: \mathrm{Ext}^{n-i}_{\mathbb{H}}(Y, \mathbb{D}(X))  \otimes \mathrm{Ext}^i_{\mathbb{H}}(X, Y) \rightarrow \mathrm{Ext}_{\mathbb{H}}^n (X, \mathbb{D}(X)) \]
is a nondegenerate pairing.
\end{enumerate}
Here the $\mathrm{Ext}^i_{\mathbb{H}}$-groups are taken in the category of $\mathbb{H}$-modules.
\end{theorem}
\noindent
(For some comments on the formulation and the proof of Theorem \ref{thm poin dua intro}, see Remark \ref{rmk comment dua result}.) Theorem \ref{thm poin dua intro} is an analogue of the Poincar\'e duality for real reductive groups (see \cite[Ch.I Proposition 2.9]{BW}, \cite[Theorem 6.10]{Kn}). Moreover, statements (1) and (2) are compatible in the sense that there exists a natural linear functional $\mathscr D_X: \mathrm{Ext}^n_{\mathbb{H}}(X, \mathbb{D}(X)) \rightarrow \mathbb{C}$ such that $\mathscr D_X \circ \mathscr Y_i$ agrees with the pairing in (1) (see Theorem \ref{thm dua YP}(3)). 


The second result of this paper is to apply the duality to compute extensions between a tempered module and a discrete series. Those tempered modules and discrete series are defined algebraically in terms of weights (Definition \ref{def temp ds}) and correspond to tempered modules and discrete series of $p$-adic groups when the parameter function is positive and equal. Since discrete series are basic building blocks of irreducible $\mathbb{H}$-modules, it may be important to first understand the extensions among them. Our second result states that:

\begin{theorem} \label{thm ds ext intro} (Theorem \ref{thm ds ext})
Let $\mathbb{H}$ be the graded affine Hecke algebra associated to a root datum $( R, V, R^{\vee}, V^{\vee}, \Pi)$ and a parameter function $k: \Pi \rightarrow \mathbb{C}$ (Definition \ref{def graded affine}). Assume $R$ spans $V$. Let $X$ be an irreducible tempered module and let $Y$ be an irreducible discrete series (Definition \ref{def temp ds}). Then
\[     \mathrm{Ext}^i_{\mathbb{H}}(X, Y) =\left\{ \begin{array}{cc}
                                                       \mathbb{C} & \quad \mbox{ if $X \cong Y$ and $i=0$ } \\
																											 0       & \quad \mbox{ otherwise . }
																					\end{array}  \right.
\]

\end{theorem}
In other words, discrete series are projective modules in the category of tempered modules. From this, one can establish an upper bound on the number of isomorphism classes of discrete series (Corollary \ref{cor dimension bound}). This is important in the classification of simple $\mathbb{H}$-modules. 

 The statement for affine Hecke algebra setting is proven by Opdam-Solleveld \cite{OS}, and the one for $p$-adic group setting is proven  by Meyer \cite{Me}, in which they used Schwartz algebras. The method we prove Theorem \ref{thm ds ext intro} is different from theirs nevertheless. Our strategy makes use of a resolution arisen from the $\mathrm{Ind}$-$\mathrm{Res}$-functors (see Section \ref{s kato resolution}) and also makes use of Theorem \ref{thm poin dua intro}. Thus we provide another point of view for the result. We remark that there is another proof for Theorem \ref{thm ds ext intro} in the author's thesis \cite[Chapter 8]{Ch}, which uses the Langland's classification and properties of parabolically induced modules.

\subsection{} This paper is organized as follows. Section \ref{s intro} is the introduction. Section \ref{s prelim} is devoted to set-up basic notations and recall the definition of the graded affine Hecke algebra. Section \ref{s resol h mod} is to construct an explicit projective resolution for $\mathbb{H}$-modules. Section \ref{s dua thm} proves a duality result for $\mathrm{Ext}$-groups by using the resolution in Section \ref{s resol h mod}. Section \ref{s dua YP} extends the duality in Section \ref{s dua thm} to the level of the Yoneda product. Section \ref{s kato resolution} studies an Ind-Res resolution and then connects the dual module induced by the involution $\mathbb{D}$ and a dual module induced from the Ind-Res resolution. Section \ref{s ext ds} computes the extensions of discrete series by using results in Sections \ref{s dua thm} and \ref{s kato resolution}. Section \ref{s epp} applies the Euler-Poincar\'e pairing to give applications.

\subsection{Acknowledgments} The author would like to thank Dan Ciubotaru, Eric Opdam, Gordan Savin and Peter Trapa for many useful discussions. The author would also like to thank Gordan Savin for several useful comments and discussions on the Aubert involution, which makes improvements for this paper. The author would also like to thank the referee for useful comments and suggestions. Part of this paper is from the author's PhD thesis \cite{Ch}. The author was supported by ERC-advanced grant no. 268105 and the Croucher Postdoctoral Fellowship during the review process.

\section{Preliminaries} \label{s prelim}
\subsection{Root systems and basic notations} \label{ss basic notation}
Let $R$ be a reduced (not necessarily crystallographic) root system. Let $\Pi$ be a fixed choice of simple roots in $R$. Then $\Pi$ determines the set $R^+$ of positive roots. Let $W$ be the finite reflection group of $R$. Let $V_0$ be a real vector space containing $R$. ($R$ does not necessarily span $V$.) For any $\alpha \in \Pi$, let $s_{\alpha}$ be the simple reflection in $W$ associated to $\alpha$ (i.e. $\alpha \in V_0$ is in the $-1$-eigenspace of $s_{\alpha}$). For $\alpha \in R$, let $\alpha^{\vee} \in \Hom_{\mathbb{R}}(V_0, \mathbb{R})$ such that
\[   s_{\alpha}(v) = v-\langle v, \alpha^{\vee} \rangle \alpha, \]
where $\langle v, \alpha^{\vee} \rangle=\alpha^{\vee}(v)$. Let $R^{\vee} \subset \Hom_{\mathbb{R}}(V_0, \mathbb{R})$ be the collection of all $\alpha^{\vee}$. Let $V_0^{\vee}= \Hom_{\mathbb{R}}(V_0, \mathbb{R})$.

By extending the scalars, let $V=\mathbb{C} \otimes_{\mathbb{R}} V_0$ and let $V^{\vee}=\mathbb{C} \otimes_{\mathbb{R}} V_0^{\vee}$. We call $(R, V, R^{\vee}, V^{\vee}, \Pi)$ to be a root datum.



\subsection{Graded affine Hecke algebras}

Let $k: \Pi \rightarrow \mathbb{C}$ be a parameter function such that $k(\alpha)=k(\alpha')$ if $\alpha$ and $\alpha'$ are in the same $W$-orbit. We shall simply write $k_{\alpha}$ for $k(\alpha)$.

\begin{definition} \label{def graded affine} \cite[Section 4]{Lu}
The graded affine Hecke algebra $\mathbb{H}=\mathbb{H}(\Pi,k)$ associated to a root datum $(R, V, R^{\vee}, V^{\vee}, \Pi)$ and a parameter function $k$ is an associative algebra with an unit over $\mathbb{C}$ generated by the symbols $\left\{ t_w :w \in W \right\}$ and $\left\{ f_w: w \in V \right\}$ satisfying the following relations:
\begin{enumerate}
\item[(1)] The map $w  \mapsto t_w$ from $\mathbb{C}[W]=\bigoplus_{w\in W} \mathbb{C}w  \rightarrow \mathbb{H }$ is an algebra injection,
\item[(2)] The map $v \mapsto f_v$ from $S(V) \rightarrow \mathbb{H}$ is an algebra injection,
\end{enumerate}
For simplicity, we shall simply write $v$ for $f_v$ from now on.
\begin{enumerate}
\item[(3)] the generators satisfy the following relation: for $\alpha \in \Pi$ and $v \in V$,
\[    t_{s_{\alpha}}v-s_{\alpha}(v)t_{s_{\alpha}}=k_{\alpha}\langle v, \alpha^{\vee} \rangle .\]
\end{enumerate}

\end{definition}

\subsection{$\mathrm{Ext}_{\mathbb{H}}$-groups and central characters}

Let $Z(\mathbb{H})$ be the center of $\mathbb{H}$. An $\mathbb{H}$-module $X$ is said to have a {\it central character} if there exists a function $\chi: Z(\mathbb{H}) \rightarrow \mathbb{C}$ such that every $z \in Z(\mathbb{H})$ acts by the scalar $\chi(z)$ on $X$. If $X$ is an irreducible $\mathbb{H}$-module, then $X$ is finite-dimensional and so has a central character by Schur's Lemma.

The center $Z(\mathbb{H})$ of $\mathbb{H}$ is $S(V)^W$ (i.e. the set of $W$-invariant polynomials) \cite[Proposition 4.5]{Lu}. Suppose $X$ is an $\mathbb{H}$-module with the central character $\chi$. Then there exists a $W$-orbit $O$ in $V^{\vee}$ such that $\chi(z)=\gamma(z)$ for any $\gamma \in O$. We shall also say the $W$-orbit $O=W\gamma$ is the central character of $X$. 

For $\mathbb{H}$-modules $X$ and $Y$, $\mathrm{Ext}^i_{\mathbb{H}}(X,Y)$ are the higher extensions of $X$ and $Y$ in the category of $\mathbb{H}$-modules. The consideration of central characters will play a crucial role in computing some $\mathrm{Ext}$-groups later. In particular, we shall use the following results later. 

\begin{proposition} \label{prop char ext 0}
Let $X$ and $Y$ be $\mathbb{H}$-modules. If $X$ and $Y$ have distinct central characters, then $\Ext^i_{\mathbb{H}}(X, Y)=0$ for all $i$. 

\end{proposition}

\begin{proof}
See for example \cite[Theorem I. 4.1]{BW}, whose proof can be modified to our setting.
\end{proof}

\section{Koszul-type resolution for $\mathbb{H}$-modules} \label{s resol h mod}

In this section, we construct an explicit projective resolution for graded affine Hecke algebra modules, which is the main tool for proving the duality for $\mathrm{Ext}$-groups in Section \ref{s dua thm}.

We keep using the notation in Section \ref{s prelim}. Let $(R, V, R^{\vee}, V^{\vee}, \Pi)$ be a root datum and let $W$ be the real reflection group associated to $R$. Let $k: \Pi \rightarrow \mathbb{C}$ be a parameter function. Let $\mathbb{H}$ be the graded affine Hecke algebra associated to $\Pi$ and $k$.


\subsection{Projective objects and injective objects}
In this section, we construct some projective objects and injective objects, which will be used to construct explicit resolutions for $\mathbb{H}$-modules in the next sections. We do not need those injective objects in this paper anyway.

Let $X$ be an $\mathbb{H}$-module and let $U$ be a finite-dimensional $\mathbb{C}[W]$-module. $\mathbb{H}$ acts on the space $\mathbb{H} \otimes_{\mathbb{C}[W]}U$ by the left multiplication on the first factor while $\mathbb{H}$ acts on the space $\mathrm{Hom}_W(\mathbb{H}, U)$ by the right translation, explicitly that is for $f \in\mathrm{Hom}_W(\mathbb{H}, U)$, the action of $h' \in \mathbb{H}$ is given by:
\[   (h'.f)(h)=f(hh') , \quad \mbox{ for all $h \in \mathbb{H}$ }.\]

Denote by $\mathrm{Res}_W$ the restriction functor from $\mathbb{H}$-modules to $\mathbb{C}[W]$-modules.

\begin{lemma} (Frobenius reciprocity) 
Let $X$ be an $\mathbb{H}$-module. Let $U$ be a $\mathbb{C}[W]$-module. Then
\[   \mathrm{Hom}_{\mathbb{H}}(X, \mathrm{Hom}_{W}(\mathbb{H}, U))=\mathrm{Hom}_{W}(\mathrm{Res}_WX, U) ,\]
and
\[ \mathrm{Hom}_{\mathbb{H}}(\mathbb{H}\otimes_{\mathbb{C}[W]}U, X)=\mathrm{Hom}_{W}(U, \mathrm{Res}_WX) .\]

\end{lemma}

\begin{proof}
Let $F :\mathrm{Hom}_{\mathfrak{R}(\mathbb{H})}(X, \mathrm{Hom}_W(\mathbb{H}, U)) \rightarrow \mathrm{Hom}_{W}(\mathrm{Res}_WX, U)$ given by 
\[(F(f))(x)=(f(x))(1). \]
 Let $G:  \mathrm{Hom}_{W}(\mathrm{Res}_WX, U) \rightarrow \mathrm{Hom}_{\mathfrak{R}(\mathbb{H})}(X, \mathrm{Hom}_W(\mathbb{H}, U))$ given by 
\[ ((G(f)(x))(h)=f(hx). \]
It is straightforward to verify $F$ and $G$ are linear isomorphisms. This proves the second equation. The proof for the second one is similar.
\end{proof}

\begin{lemma} \label{lem proj inj obj}
Let $U$ be a $\mathbb{C}[W]$-module. Then $\mathbb{H} \otimes_{\mathbb{C}[W]}U$ is projective and $\mathrm{Hom}_{W}(\mathbb{H}, U)$ is injective.

\end{lemma}
\begin{proof}
We consider $\mathbb{H} \otimes_{\mathbb{C}[W]}U$. Every $\mathbb{C}[W]$-module is projective and so $\mathrm{Hom}_W(U, .)$ is an exact functor. The functor $\mathrm{Res}_W$ is also exact. Thus we have $\mathrm{Hom}_{W}(U, \mathrm{Res}_W . )$, which is the composition of the two functors $\mathrm{Hom}_W(U,.)$ and $\mathrm{Res}_W$, is also exact. Hence by the Frobenius reciprocity, we also have $\mathrm{Hom}_{\mathbb{H}}(\mathbb{H} \otimes_{\mathbb{C}[W]}U,.)$ is exact. Thus $\mathbb{H} \otimes_{\mathbb{C}[W]}U$ is projective. The proof for $\mathrm{Hom}_{W}(\mathbb{H}, U)$ being injective is similar.
\end{proof}

\subsection{Koszul-type resolution on $\mathbb{H}$-modules} \label{sec kos resol}


Let $X$ be an $\mathbb{H}$-module. Define a sequence of $\mathbb{H}$-maps $d_i$ as follows:
\begin{align}  \nonumber 0 \rightarrow \mathbb{H} \otimes_{\mathbb{C}[W]}(\Res_WX \otimes \wedge^n V)   \nonumber
      \stackrel{d_{n-1}}{\rightarrow} \ldots & \stackrel{d_{i}}{\rightarrow} \mathbb{H} \otimes_{\mathbb{C}[W]}(\Res_WX \otimes \wedge^i V)  \\  
		 \label{eqn projective resolution}
			 & \stackrel{d_{i-1}}{\rightarrow} \ldots  \stackrel{d_0}{\rightarrow} \mathbb{H} \otimes_{\mathbb{C}[W]}(\Res_WX) \stackrel{\epsilon}{\rightarrow}  X \rightarrow 0
\end{align}
such that $\epsilon: \mathbb{H} \otimes_{\mathbb{C}[W]}\Res_WX \rightarrow X$ given by
\[  \epsilon(h \otimes x)= h.x
\]
and for $i \geq 1$, $d_{i}: \mathbb{H} \otimes_{\mathbb{C}[W]} (\Res_WX \otimes \wedge^{i+1} V) \rightarrow  \mathbb{H} \otimes_{\mathbb{C}[W]} (\Res_WX \otimes \wedge^{i} V)$ given by
 \begin{align*}  
&d_{i}(h\otimes (x \otimes v_1 \wedge\ldots \wedge v_{i+1}))  \\
= & \sum_{j=1}^{i+1} (-1)^{j+1}(h v_j \otimes x \otimes v_1 \wedge \ldots \wedge \widehat{v}_j \wedge \ldots \wedge v_{i+1} -h \otimes v_j. x \otimes v_1 \wedge \ldots \wedge \widehat{v}_j \wedge \ldots \wedge v_{i+1} ) .
\end{align*}
If we want to emphasize the role of $X$, we shall write $d_{i,X}$ for $d_i$. We also write $d_{-1}$ for $\epsilon$.
In priori, we do not know $d_i$ is a well-defined $\mathbb{H}$-map, but we prove in the following:

\begin{lemma} \label{lem well defined d}
The above $d_i$ are well-defined $\mathbb{H}$-maps and $d^2=0$ i.e. (\ref{eqn projective resolution}) is a well-defined complex.
\end{lemma}

\begin{proof}
We proceed by induction on $i$. It is easy to see that $\epsilon$ is well-defined. For convenience, we set $d_{-1}=\epsilon$. We now assume $i \geq 0$. To show $d_i$ is independent of the choice of a representative in $\mathbb{H} \otimes_{\mathbb{C}[W]}(\mathrm{Res}_WX \otimes \wedge^{i+1} V)$, the only non-trivial thing is to show
\begin{align} \label{eqn di map}
    d_i(t_w \otimes (x \otimes v_1\wedge \ldots \wedge v_{i+1})) =d_i(1 \otimes (t_w.x \otimes w(v_1)\wedge \ldots \wedge w(v_{i+1}))).
\end{align}
For simplicity, set
\begin{eqnarray*}
  P^w &=& d_i(t_w \otimes (x \otimes v_1\wedge \ldots \wedge v_{i+1}))  \\
      &=& t_w\sum_{j=1}^{i+1} (-1)^{j+1}(v_j \otimes (x \otimes v_1 \wedge \ldots \wedge \widehat{v}_j \wedge \ldots \wedge v_{i+1} )-  1 \otimes (v_j.x \otimes v_1\wedge \ldots \wedge \widehat{v}_j \wedge \ldots \wedge v_{i+1} ))
\end{eqnarray*}
and
\begin{eqnarray*}
  P_w &=& d_i(1 \otimes (t_w.x \otimes w(v_1)\wedge \ldots \wedge w(v_{i+1}))  \\
      &=& \sum_{j=1}^{i+1} (-1)^{j+1}w(v_j) \otimes (t_w.x \otimes w(v_1) \wedge \ldots \wedge \widehat{w(v_j)} \wedge \ldots \wedge w(v_{i+1})) \\
			& & \quad - \sum_{j=1}^{i+1} (-1)^{j+1}  \otimes (w(v_j)t_w.x \otimes w(v_1)\wedge \ldots \wedge \widehat{w(v_j)} \wedge \ldots \wedge w(v_{i+1})
\end{eqnarray*}
To show the equation (\ref{eqn di map}), it is equivalent to show $P^w=P_w$. Regard $\mathbb{C}[W]$ as a natural subalgebra of $\mathbb{H}$. By using the fact that $t_wv-w(v)t_w \in \mathbb{C}[W]$ for $w \in W$,  $P^w -P_w$ is an element of the form $1 \otimes u$ for some $u \in \mathrm{Res}_WX \otimes \wedge^i V$. Thus it suffices to show that $u=0$. To this end, a direct computation (from the original expressions of $P^w$ and $P_w$) shows that $d_{i-1}(P^w-P_w)=0$. By induction hypothesis, $d_{i-1}$ is well-defined and so $d_{i-1}(1 \otimes u)=d_{i-1}(P^w-P_w)=0$. Write $1 \otimes u$ of the form
\begin{align}\label{eqn expression d} 
1 \otimes u & = \sum_{1 \leq r_1 <  \ldots  < r_{i} \leq n}  1 \otimes (x_{r_1,\ldots, r_{i}}) \otimes e_{r_1}\wedge \ldots  \wedge e_{r_{i}},
\end{align}
where $x_{r_1,\ldots, r_{i}} \in \mathrm{Res}_WX$ and $e_1, \ldots, e_n$ is a fixed basis of $V$. By a direct computation of $d_{i-1}(1 \otimes u)$ from the expression (\ref{eqn expression d}), we have
\begin{align*}
& d_{i-1}(1 \otimes u) \\
 = &     \sum_{1 \leq r_1 <  \ldots < r_{i} \leq n} \sum_{j=1}^{i}(-1)^{j+1} e_{r_j} \otimes (x_{r_1,\ldots, r_{i}}) \otimes e_{r_1}\wedge \ldots \wedge \widehat{e}_{r_j} \wedge \ldots \wedge e_{r_{i}})  \\
                     & \quad -\sum_{1 \leq r_1 < \ldots < r_{i} \leq n}\sum_{j=1}^{i}(-1)^{j+1} 1 \otimes  e_{r_j}.(x_{r_1,\ldots, r_{i}}) \otimes e_{r_1}\wedge \ldots \wedge \widehat{e}_{r_j} \wedge \ldots \wedge e_{r_{i}})
\end{align*}
We have seen that $d_{i-1}(1\otimes u)=0$ and so $u=0$ by using linearly independence arguments.

Verifying $d^2=0$ is straightforward. 
\end{proof}

\begin{theorem} \label{thm projective resol}
\begin{enumerate}
\item[(1)] For any $\mathbb{H}$-module $X$, the complex (\ref{eqn projective resolution}) forms a projective resolution for $X$.
\item[(2)] The global dimension of $\mathbb{H}$ is $\dim V$.
\end{enumerate}
\end{theorem}

\begin{proof}  

 From Lemma \ref{lem well defined d}, it remains to show the exactness for (1). This can be proven by an argument which imposes a filtration on $\mathbb{H}$ and uses a long exact sequence (see for example \cite[Section 5.3.8]{HP} or \cite[Chapter IV Section 6]{Kn}). We provide some detail. Let $\mathbb{H}^r$ be the (vector) subspace of $\mathbb{H}$ spanned by the elements of the form
\[        t_w v_1^{n_1}\ldots v_l^{n_l} \quad \mbox{ for $w \in W$, $v_1, \ldots, v_l \in V$ },\]
with $n_1 +n_2+\ldots+n_l \leq r$. Note that $\mathbb{H}^r$ is still (left and right) invariant under the action of $W$. Let
\[  \mathbb{E}^{r,s} = \mathbb{H}^r \otimes_{\mathbb{C}[W]}  (\mathrm{Res}_W X \otimes \wedge^s V) .\]
Then the differential $d_{s-1}$ defines a map from $\mathbb{E}^{r,s}$ to $\mathbb{E}^{r+1,s-1}$. For convenience, also set $\mathbb{E}^{r+s+1,-1}=X$ and there is a map from $\mathbb{E}^{r+s,0}$ to $X$ and $d_{-1}=\epsilon$. Then for a positive integer $p$, we denote by $\mathcal E(p)$ the complex  $\left\{ \mathbb{E}^{r,s}, d_{s-1} \right\}_{r+s=p, s\geq -1}$. We now define a graded structure. Let
\[  \mathbb{F}^{r,s}=\mathbb{E}^{r,s}/ \mathbb{E}^{r-1,s} .
\]
and let $\overline{d}_{r,s}:\mathbb{F}^{r,s} \rightarrow \mathbb{F}^{r+1, s-1} $ be the induced map from $d_{s-1}$. Then $\left\{ \mathbb{F}^{r,s}, \overline{d}_{r,s} \right\}_{r+s=p}$ forms a complex for each $p$. Denote by $\mathcal F(p)$ for such complex. In fact, $\mathcal F(p)$ forms a standard Koszul complex and hence the homology $H^i(\mathcal F(p))=0$ for all $i$.

Now consider the following short exact sequences of the chain of complexes for $p \geq 1$:

\[\xymatrix{  & 0 \ar[d]  & 0 \ar[d]& 0\ar[d]  &  &   \\
0 & X \ar[d] \ar[l] & \mathbb{E}^{p-1,0} \ar[d]\ar[l]& \mathbb{E}^{p-2,1}\ar[d] \ar[l] & \cdots \ar[l] &   \\
0  & X\ar[d] \ar[l] & \mathbb{E}^{p,0} \ar[d]\ar[l] & \mathbb{E}^{p-1,1} \ar[d] \ar[l] & \cdots \ar[l] & \\
              0 & 0  \ar[l] \ar[d]     & \mathbb{F}^{p,0} \ar[l] \ar[d]  & \mathbb{F}^{p-1,0}     \ar[d]         \ar[l] & \cdots       \ar[l] &   \\
							 & 0       & 0         & 0              &  &   }
\]
The vertical map from $\mathbb{E}^{p-s-1,s}$ to  $\mathbb{E}^{p-s,s}$ is the natural inclusion map. Then we have the associated long exact sequence:
\[   \ldots \rightarrow H^{k+1}(\mathcal F(p)) \rightarrow H^k(\mathcal E(p-1))\rightarrow H^k(\mathcal E(p)) \rightarrow H^k(\mathcal F(p)) \rightarrow \ldots 
\]
Since $ H^k(\mathcal F(p))=H^{k+1}(\mathcal F(p))=0$, $H^k(\mathcal E(p-1))\cong H^k(\mathcal E(p))$. It remains to see $H^k(\mathcal E(0))=0$ for all $k$, but it follows from definitions.

We now prove (2). By (1), the global dimension of $\mathbb{H}$ is less than or equal to $\dim V$. It remains to prove the global dimension attains the upper bound. Let $\gamma \in V^{\vee}$ be a regular element and let $v_{\gamma}$ be a vector with weight $\gamma \in V^{\vee}$. Define $X=\Ind_{S(V)}^{\mathbb{H}}\mathbb{C}v_{\gamma}$. By Frobenius reciprocity and using $\gamma$ is regular, $\Ext_{\mathbb{H}}^i(X,X)=\Ext_{S(V)}^i(\mathbb{C}v_{\gamma}, \mathbb{C}v_{\gamma})\neq 0$ for all $i \leq \dim V$. This shows the global dimension has to be $\dim V$. 
\end{proof}


\subsection{Alternate form of the Koszul-type resolution}

In this section, we give another form of the differential map $d_i$, which involves the terms $\widetilde{v}$ defined in (\ref{eqn v titlde}) below. There are some advantages for computations later. 

For $v \in V$, we define the following element in $\mathbb{H}$:
\begin{eqnarray} \label{eqn v titlde}
    \widetilde{v}&= v- \frac{1}{2} \sum_{\alpha \in R^+} k_{\alpha} \langle v, \alpha^{\vee} \rangle t_{s_{\alpha}} .
\end{eqnarray}
This element is used by Drinfield \cite{Dr} for the study of Yangians and also used by Barbasch-Ciubotaru-Trapa \cite{BCT} for the Dirac cohomology for graded affine Hecke algebras. An important property of the element is the following, which can be proved by an induction on the length of $w \in W$:

\begin{lemma} \cite[Proposition 2.10]{BCT} \label{lem tilde element}
 For any $w \in W$ and $v \in V$, $t_w\widetilde{v}=\widetilde{w(v)}t_w$.
\end{lemma}



We consider the maps
$\widetilde{d}_i: \mathbb{H} \otimes_{\mathbb{C}[W]} (\Res_WX \otimes \wedge^{i+1} V) \rightarrow \mathbb{H} \otimes_{\mathbb{C}[W]}(\Res_WX \otimes \wedge^{i}V)$ as follows:
\begin{align}
 &  \widetilde{d}_i(h\otimes (x \otimes v_1 \wedge\ldots \wedge v_{i+1})) \\
=&\sum_{j=1}^{i+1} (-1)^{j+1}\left( h \widetilde{v}_j \otimes x \otimes v_1 \wedge \ldots \widehat{v}_j \ldots \wedge v_i -h \otimes \widetilde{v}_j. x \otimes v_1 \wedge \ldots \widehat{v}_j \ldots \wedge v_{i+1} \right) .
\end{align}
This definition indeed coincides with the one in the previous subsection:

\begin{proposition} \label{prop equal d}
$\widetilde{d}_i=d_i$. 
\end{proposition}

\begin{proof}
 Recall that for $v \in V$,
\[  \widetilde{v} =v-\frac{1}{2}\sum_{\alpha \in R^+} k_{\alpha} \langle  v, \alpha^{\vee} \rangle t_{s_{\alpha}} .\]
Then 
\begin{eqnarray*}
& &\widetilde{v}_{r} \otimes (x \otimes v_{1} \wedge \ldots \wedge \widehat{v}_{r} \wedge \ldots \wedge v_{i+1}) -1 \otimes (\widetilde{v}_r.x \otimes v_1 \wedge \ldots \wedge \widehat{v}_r \wedge \ldots \wedge v_{i+1}) \\
&=&v_{r} \otimes (x \otimes v_{1} \wedge \ldots \wedge \widehat{v}_{r} \wedge \ldots \wedge v_{i+1}) -1 \otimes (v_r. x \otimes v_1 \wedge \ldots \wedge \widehat{v}_r \wedge \ldots \wedge v_{i+1})\\
& & \quad -\frac{1}{2}\sum_{\alpha \in R^+}k_{\alpha}\langle  v_{r}, \alpha^{\vee} \rangle  \otimes (t_{s_{\alpha}}.x) \otimes s_{\alpha}(v_{1})\wedge \ldots \wedge s_{\alpha}(\widehat{v}_{r})\wedge \ldots \wedge s_{\alpha}(v_{i+1}) \\
& & \quad +\frac{1}{2}\sum_{\alpha \in R^+} k_{\alpha} \langle  v_{r}, \alpha^{\vee} \rangle  \otimes (t_{s_{\alpha}}.x)  \otimes v_{1}\wedge \ldots \wedge \widehat{v}_{r} \wedge \ldots \wedge v_{i+1} \\
&=& v_{r} \otimes (x \otimes v_{1} \wedge \ldots \wedge \widehat{v}_{r} \wedge \ldots \wedge v_{i+1})-1\otimes (v_{r}.x \otimes v_{1} \wedge \ldots \wedge \widehat{v}_{r} \wedge \ldots \wedge v_{i+1})\\
& & -\frac{1}{2}\sum_{\alpha \in R^+} \sum_{p < r}(-1)^{p} k_{\alpha}\langle  v_{r}, \alpha^{\vee} \rangle \langle  v_{p}, \alpha^{\vee}  \rangle  \otimes (t_{s_{\alpha}}.x) \otimes \alpha \wedge s_{\alpha}(v_{1})\wedge \ldots  s_{\alpha}(\widehat{v}_{p})\wedge \ldots s_{\alpha}(\widehat{v}_{r})\wedge \ldots \wedge s_{\alpha}(v_{i+1}) \\
& & -\frac{1}{2}\sum_{\alpha \in R^+} \sum_{ r<p} (-1)^{p-1}k_{\alpha}\langle  v_{r}, \alpha^{\vee} \rangle \langle  v_{p}, \alpha^{\vee}  \rangle  \otimes (t_{s_{\alpha}}.x)\otimes \alpha \wedge s_{\alpha}(v_1)\wedge \ldots  s_{\alpha}(\widehat{v}_{r})\wedge \ldots s_{\alpha}(\widehat{v}_{p})\wedge \ldots \wedge s_{\alpha}(v_{i+1}) \\
\end{eqnarray*}
The second equality follows from the expression of $\widetilde{v}_r$. Taking the alternating sum of the above expression with some standard computations can verify $\widetilde{d}_i=d_i$. 
\end{proof}

\subsection{Complex for computing $\mathrm{Ext}$-groups} \label{sec complex ext}

We now use the resolution in Section \ref{sec kos resol} to construct a complex for computing $\mathrm{Ext}$-groups. Let $X$ and $Y$ be $\mathbb{H}$-modules.

Then taking the $\mathrm{Hom}_{\mathbb{H}}(.,Y)$ functor on the projective resolution of $X$ as the one in (\ref{eqn projective resolution}), we have the induced maps
\[ d_i^{\vee} : \mathrm{Hom}_{\mathbb{H}}(\mathbb{H} \otimes_{\mathbb{C}[W]} (\mathrm{Res}_W X \otimes^iV), Y) \rightarrow \mathrm{Hom}_{\mathbb{H}}(\mathbb{H} \otimes_{\mathbb{C}[W}(\mathrm{Res}_WX \otimes^{i+1}V), Y)  .
\] 
We write $d_{i,X}^{\vee,Y}$ for $d_i^{\vee}$ if we need to emphasize the roles of $X$ and $Y$. Using Frobenius reciprocity, we have the induced maps,
\[d_{i,W}^{\vee}: \mathrm{Hom}_{\mathbb{H}}(\mathrm{Res}_W X \otimes \wedge^{i} V, \mathrm{Res}_WY) \rightarrow \mathrm{Hom}_{\mathbb{H}}(\mathrm{Res}_W X \otimes \wedge^{i+1} V, \mathrm{Res}_WY) .
\]
We again write $d_{i, X, W}^{\vee, Y}$ for $d_{i, W}^{\vee}$ if we need to emphasize the roles of $X$ and $Y$. Then by using the Frobenius reciprocity, we have an induced complex 
\begin{align}  \nonumber 0 \leftarrow \mathrm{Hom}_W(\Res_WX \otimes \wedge^n V, \mathrm{Res}_WY)   \nonumber
      \stackrel{d_{n-1, W}^{\vee}}{\leftarrow} \ldots & \stackrel{d_{i, W}^{\vee}}{\leftarrow} \mathrm{Hom}_{W}(\Res_WX \otimes \wedge^i V, \mathrm{Res}_WY)  \\  
		 \label{eqn resolution hom}
			 & \stackrel{d_{i-1, W}^{\vee}}{\leftarrow} \ldots  \stackrel{d_{0,W}^{\vee}}{\leftarrow} \mathrm{Hom}_{W}(\Res_WX, \mathrm{Res}_WY)  \stackrel{}{\leftarrow}  0,
\end{align}
where the map $d_{i,W}^{\vee}$ can be explicitly written as:
\begin{align*}  d_{i,W}^{\vee}(\eta)(x \otimes v_1 \wedge \dots \wedge v_{i+1})	 &=\sum_{j=1}^{i+1} (-1)^{j+1}\widetilde{v}_j.\eta(x \otimes v_1 \wedge \ldots \wedge \widehat{v}_j\wedge \ldots \wedge v_{i+1})  \\
  &    \quad \quad -\sum_{j=1}^{i+1} (-1)^{j+1}\eta(\widetilde{v}_j.x \otimes v_1 \wedge \ldots \wedge \widehat{v}_j\wedge \ldots \wedge v_{i+1}) ,
\end{align*}
where the action of $\widetilde{v}_j$ on the term $\eta(x \otimes v_1 \wedge \ldots \wedge \widehat{v}_j\wedge \ldots \wedge v_{i+1})$ is via the action of $\widetilde{v}_j$ on $Y$ and the action of $\widetilde{v}_j$ on the term $x$ is via the action of $\widetilde{v}_j$ on $X$.

Thus we obtain the following:
\begin{proposition} \label{prop complex ext gp}
Let $X$ and $Y$ be $\mathbb{H}$-modules. Then $\mathrm{Ext}^i_{\mathbb{H}}(X, Y)$ is naturally isomorphic to the $i$-th homology of the complex in (\ref{eqn resolution hom}).

\end{proposition}

An immediate consequence is the following:

\begin{corollary} \label{cor finite dim ext}
Let $X$ and $Y$ be finite-dimensional $\mathbb{H}$-modules. Then
\[ \mathrm{dim} \mathrm{Ext}^i_{\mathbb{H}}(X, Y) < \infty. \]
\end{corollary}


It is not hard to see that $d_{i,W}^{\vee}$ can be naturally extended to a map from $\mathrm{Hom}_{\mathbb{C}}(\Res_W X \otimes \wedge^iV, \mathrm{Res}_WY)$ to $\mathrm{Hom}_{\mathbb{C}}(\Res_WX \otimes \wedge^{i+1}V, \mathrm{Res}_WY)$. We denote the map by $\overline{d}_i^{\vee}$, which will be used in Section \ref{sec complex duals}.

\section{Duality for $\mathrm{Ext}$-groups} \label{s dua thm}

In this section, we prove a duality result for the $\mathrm{Ext}$-groups of graded affine Hecke algebra modules, which is an analogue of some classical dualities such as Poincar\'e duality or Serre duality (also see Poincar\'e duality for real reductive groups in \cite[Theorem 6.10]{Kn}). 

We keep using the notation from Section \ref{s resol h mod}.

\subsection{$\theta$-action and $\theta$-dual} \label{ss theta action}
We define an involution $\theta$ on $\mathbb{H}$ in this section. This $\theta$ is not needed in stating the duality result (Theorem \ref{thm poin dua}), but it closely relates to the $*$ and $\bullet$ operations defined in the next section.

Let $w_0$ be the longest element in $W$. Let $\theta$ be an involution on $\mathbb{H}$ characterized by
\begin{eqnarray}
\label{eqn involution} 
\theta(v)=-w_0(v) \mbox{ for any $v \in V$}, \mbox{ and } \quad \theta(t_w)=t_{w_0ww_0^{-1}} \mbox{ for any $w \in W$} ,
\end{eqnarray}
where $w_0$ acts on $v$ as the action on the reflection representation of $W$. Since $\theta(\Pi)=\Pi$, $\langle ., .\rangle$ is $W$-invariant and $k_{\alpha}=k_{\theta(\alpha)}$ for any $\alpha \in \Pi$, it is straightforward to verify $\theta$ defines an automorphism on $\mathbb{H}$.

Note that $\theta$ also induces an action on $V^{\vee}$, still denoted $\theta$. For $\alpha \in R$, since $w_0(\alpha^{\vee})=w_0(\alpha)^{\vee}$, we also have $\theta(\alpha^{\vee})=\theta(\alpha)^{\vee}$.

Recall that for $v \in V$, $\widetilde{v}$ is defined in (\ref{eqn v titlde}). The following lemma follows from the definitions.
\begin{lemma} \label{lem theta v tilde}
For any $v \in V$, $\theta(\widetilde{v})=\widetilde{\theta(v)}$.
\end{lemma}


\begin{definition}
For an $\mathbb{H}$-module $X$, define $\theta(X)$ to be the $\mathbb{H}$-module such that $\theta(X)$ is isomorphic to $X$ as vector spaces and the $\mathbb{H}$-action is determined by:
\[   \pi_{\theta(X)}(h)x=\pi_X(\theta(h))x ,\]
where $\pi_X$ and $\pi_{\theta(X)}$ are the maps defining the action of $\mathbb{H}$ on $X$ and $\theta(X)$ respectively.
\end{definition}

\subsection{*-dual and $\bullet$-dual} \label{sec duals}

In this section, we study two anti-involutions on $\mathbb{H}$. These two anti-involutions are studied in \cite{BC}, but we make a slight variation for our need. More precisely, those anti-involutions are linear rather than Hermitian-linear. The linearity will make some construction easier. For instance, it is easier to make the identification of spaces in Section \ref{sec complex duals}.

Define $* : \mathbb{H} \rightarrow \mathbb{H}$ to be the linear anti-involution determined by 
\[    v^*=t_{w_0} \theta(v)t_{w_0}^{-1} \quad \mbox{ for $v \in V$} , \quad t_w^*=t_w^{-1} \quad \mbox{ for $w \in W$ }. 
\]

Define $\bullet: \mathbb{H} \rightarrow \mathbb{H}$ to be another linear anti-involution determined by
\[  v^{\bullet}=v \quad \mbox{ for $v \in V$ }, \quad t_w^{\bullet}=t_w^{-1} \quad \mbox{ for $w \in W$ }.
\]

\begin{definition} \label{def star bullet duals}
Let $X$ be an $\mathbb{H}$-module. A map $f: X \rightarrow \mathbb{C}$ is said to be a linear functional if $f(\lambda x_1+x_2)=\lambda f(x_1)+f(x_2)$ for any $x_1, x_2 \in X$ and $\lambda \in \mathbb{C}$. 
 The {\it $*$-dual of $X$}, denoted by $X^*$, is the space of linear functionals of $X$ with the action of $\mathbb{H}$ determined by
\begin{align} \label{eqn star hermitian}
  (h.f)(x)=f(h^*.x) \quad \mbox{ for any $x \in X$ }  . 
\end{align}
We similarly define {\it $\bullet$-dual} of $X$, denoted by $X^{\bullet}$, by replacing $h^*$ with $h^{\bullet}$ in equation (\ref{eqn star hermitian}).

When $X$ is a finite-dimensional $\mathbb{H}$-module, we have $(X^*)^*=X$ and $(X^{\bullet})^{\bullet}=X$.
\end{definition}

\begin{lemma} \label{lem star form}
Let $X$ be an $\mathbb{H}$-module. Define a bilinear pairing $\langle , \rangle_X^*: X^* \times X \rightarrow \mathbb{C}$ (resp. ${}^{\bullet}\langle , \rangle_X: X^{\bullet} \times X \rightarrow \mathbb{C}$) such that $\langle f, x \rangle^*_X=f(x)$ (resp. ${}^{\bullet}\langle f, x \rangle_X=f(x)$). (We reserve $\langle , \rangle^{\bullet}_X$ for the use of another pairing later.) Then 
\begin{enumerate}
\item[(1)] for $v \in V$, $\langle \widetilde{v}.f, x \rangle^*_X = \langle f, -\widetilde{v}.x \rangle^*_X$ (resp. ${}^{\bullet}\langle \widetilde{v}.f, x \rangle_X ={}^{\bullet}\langle f, \widetilde{v}.x \rangle_X$),
\item[(2)] for $w \in W$, $\langle t_w.f, x \rangle^*_X = \langle f, t_w^{-1}.x \rangle^*_X$ (resp. ${}^{\bullet}\langle t_w.f, x \rangle_X = {}^{\bullet}\langle f, t_w^{-1}.x \rangle_X$),
\item[(3)] $\langle , \rangle^*_X$ (resp. ${}^{\bullet}\langle , \rangle_X$) is non-degenerate. 
\end{enumerate}
\end{lemma}

\begin{proof}
We first consider the $*$-operation. Note that $(\widetilde{v})^*=t_{w_0}\theta(\widetilde{v})t_{w_0}^{-1}=\widetilde{w_0\theta(v)}=-\widetilde{v}$, where the second equality follows from Lemma \ref{lem tilde element} and Lemma \ref{lem theta v tilde}. This implies (1). Other assertions follow immediately from the definitions. 

For the $\bullet$-operation, we have $(\widetilde{v})^{\bullet}=\widetilde{v}$ from the definitions. This implies (1). Other assertions again follow from the definitions.
\end{proof}

\begin{lemma} \label{lem two dual theta}
Let $X$ be an $\mathbb{H}$-module. Then $X^* \cong \theta(X)^{\bullet}$. 
\end{lemma}

\begin{proof}
We define a map $F: \theta(X)^{\bullet} \rightarrow X^*$, $f \mapsto t_{w_0}^{-1}.f$ (where $t_{w_0}$ acts on $f$ by the action of $\mathbb{H}$ on $\theta(X)^{\bullet}$. In order to distinguish the action of $\mathbb{H}$ on the modules $X^*$ and $\theta(X)^{\bullet}$, we use the $\pi_{X^*}:\mathbb{H} \rightarrow \mathrm{End}_{\mathbb{C}}(X^*)$ and $\pi_{\theta(X)^{\bullet}}:\mathbb{H} \rightarrow \mathrm{End}_{\mathbb{C}}(\theta(X)^{\bullet})$ to denote the corresponding actions of $\mathbb{H}$ respectively. Now it is straightforward to verify: for $w \in W$, 
\begin{align*}
 F(\pi_{X^*}(t_w)f)(x) &=(\pi_{\theta(X)^{\bullet}}(t_{w_0})\pi_{X^*}(t_w)f)(x)  \\
                       &= f(t_{w_0}t_w.x) \\
											 &= (\pi_{\theta(X)^{\bullet}}(t_w)\pi_{\theta(X)^{\bullet}}(t_{w_0})f)(x) \\
											 &= (\pi_{\theta(X)^{\bullet}}(t_w)F(f))(x) \qedhere
\end{align*}
\end{proof}
\subsection{Pairing for $\wedge^i V$ and $\wedge^{n-i} V$} \label{ss pairing wedge V}
Fix an ordered basis $e_1, \ldots, e_n$ for $V$. We define a non-degenerate bilinear pairing $\langle ,\rangle_{\wedge^i V}$ as 
\[  \wedge^i V  \times \wedge^{n-i} V \rightarrow \mathbb{C}
\]
determined by 
\[  \langle v_1 \wedge \ldots \wedge v_i, v_{i+1} \wedge \ldots \wedge v_n \rangle_{\wedge^i V} =\det(v_1 \wedge \ldots \wedge v_n) ,\]
where $\det$ is the determinant function for the fixed ordered basis $e_1, \ldots, e_n$. 

Define $(\wedge^iV)^{\vee}$ to be the dual space of $\wedge^iV$. For $\omega \in \wedge^{n-i}V$, define $\phi_{\omega} \in (\wedge^iV)^{\vee}$ by
\begin{align} \label{eqn map wedge v}
 \phi_{\omega}(\omega')=\langle \omega, \omega' \rangle_{\wedge^{n-i}V} 
\end{align}
By using $\det(w(v_1) \wedge \ldots \wedge w(v_n))=\sgn(w)\det(v_1 \wedge \ldots \wedge v_n)$ for any $w \in W$, we see the map $\omega \mapsto \phi_{\omega}$ from $\wedge^{n-i}V$ to $(\wedge^iV)^{\vee}$ defines a $W$-representation isomorphism from $\wedge^{n-i}V$ to $\sgn \otimes (\wedge^iV)^{\vee}$.

We also define a pairing $\langle , \rangle_{(\wedge^{n-i}V)^{\vee}}: (\wedge^{n-i}V)^{\vee} \times (\wedge^iV)^{\vee} \rightarrow \mathbb{C}$ such that
\[  \langle \phi_{\omega}, \phi_{\omega'} \rangle_{(\wedge^{n-i}V)^{\vee}}=\langle \omega, \omega' \rangle_{\wedge^iV} , 
\]
where $\omega \in \wedge^i V$ and $\omega' \in \wedge^{n-i} V$. By the definitions, we have the following two lemmas:
\begin{lemma} \label{lem equi inner wedge form}
For $\omega \in \wedge^iV$ and $\omega' \in \wedge^{n-i}V$,
\[  \langle \phi_{\omega}, \phi_{\omega'} \rangle_{_{(\wedge^{n-i}V)^{\vee}}} = \langle \omega, \omega' \rangle_{\wedge^i V}=\phi_{\omega}(\omega') .\]
\end{lemma}

\begin{lemma} \label{lem property wedge form}
For any $w \in W$, 
\[\langle w. \phi_{v_1 \wedge \ldots \wedge v_i}, w. \phi_{v_{i+1} \wedge \ldots \wedge v_n} \rangle_{(\wedge^{n-i}V)^{\vee}}=\sgn(w) \langle \phi_{v_1\wedge \ldots \wedge v_i}, \phi_{v_{i+1} \wedge \ldots \wedge v_n} \rangle_{(\wedge^{n-i}V)^{\vee}}. \]
\end{lemma}

\begin{proof}
As noted above, for any $w\in W$,
\[w. \phi_{v_1 \wedge \ldots \wedge v_i}=\sgn(w)\phi_{w.(v_1 \wedge \ldots \wedge v_i)} ,\quad w. \phi_{v_{i+1} \wedge \ldots \wedge v_n}=\sgn(w)\phi_{w.(v_{i+1} \wedge \ldots \wedge v_n)}.\] 
The lemma then follows from straightforward computations.
\end{proof}

\subsection{Complexes involving duals} \label{sec complex duals}
Let $X$ and $Y$ be finite dimensional $\mathbb{H}$-modules. It is well-known that there is a natural identification between the spaces $\mathrm{Hom}_{W}(\mathrm{Res}_{W}X \otimes \wedge^iV, \mathrm{Res}_WY)$ and $(X^* \otimes Y \otimes (\wedge^iV)^{\vee})^W$ (or $(X^{\bullet} \otimes Y \otimes (\wedge^iV)^{\vee})^W$). Here we consider a natural $W$-action on $X^* \otimes Y \otimes (\wedge^iV)^{\vee}$, and $(X^* \otimes Y \otimes (\wedge^iV)^{\vee})^W$ is the $W$-invariant space.  For notational simplicity, we may simply write $X$ for $\mathrm{Res}_WX$ and also sometimes regard $X$ as a vector space, which should be clear from the context.

In order to prove Theorem \ref{thm poin dua} later, we need to construct some pairings, which will be more convenient to be done for the spaces $(X^* \otimes Y \otimes (\wedge^iV)^{\vee})^W$  (or $(X^{\bullet} \otimes Y \otimes (\wedge^iV)^{\vee})^W$). The goal of this section is to translate the differential maps $d_{i}^{\vee}$ in Section \ref{sec complex ext} into the corresponding maps for $(X^* \otimes Y \otimes (\wedge^iV)^{\vee})^W$.

We define a (linear) map $\overline{D}_i:X^* \otimes Y\otimes (\wedge^iV)^{\vee} \rightarrow X^* \otimes Y\otimes (\wedge^{i+1}V)^{\vee} $ on the complex, which is determined by 
\begin{align} 
\nonumber & \overline{D}_i( f \otimes y \otimes \phi_{ v_{1} \wedge \ldots \wedge v_{n-i}}) \\
=& \sum_{j=1}^{n-i}(-1)^{j+1}(f\otimes \widetilde{v}_{j}.y \otimes \phi_{v_{1} \wedge \ldots \widehat{v}_{j} \ldots \wedge v_{n-i}}+\widetilde{v}_{j}.f\otimes y \otimes \phi_{v_{1} \wedge \ldots \widehat{v}_{j} \ldots \wedge v_{n-i}}),\label{eqn D_i def}
\end{align}
for $f \otimes y \otimes \phi_{v_{1} \wedge \ldots \wedge v_{n-i}} \in X^* \otimes Y\otimes (\wedge^iV)^{\vee} $, where $\widetilde{v}_{j}$ acts on $f$ by the action on $X^*$ and $\widetilde{v}_{j}$ acts on $y$ by the action on $Y$. 

 Note that there is a natural $W$-action on $X^* \otimes Y \otimes (\wedge^iV)^{\vee}$ (from the $W$-action of $X^*$, $Y$ and $V$). Such $W$ action commutes with $\overline{D}_i$ and so $\overline{D}_i$ sends $(X^* \otimes Y\otimes (\wedge^iV)^{\vee})^W$ to $(X^* \otimes Y\otimes (\wedge^{i+1}V)^{\vee} )^W$.  Then the map $\overline{D}_i$ gives rise to the following map:
\[ D_i: (X^* \otimes Y\otimes (\wedge^iV)^{\vee})^W \rightarrow (X^* \otimes Y\otimes (\wedge^{i+1}V)^{\vee})^W \] 
 If we want to emphasize the complexes that $\overline{D}_i$ or $D_i$ refer to, we shall write $\overline{D}_{ X^* \otimes Y \otimes (\wedge^iV)^{\vee}}$ for $\overline{D}_i$ and $D_{(X^* \otimes Y \otimes (\wedge^i V)^{\vee})^W}$ for $D_i$. In the priori, we do not have $D^2=0$, but we will soon prove it in Lemma \ref{lem D d commute}.

We define another map $\overline{D}_i^{\bullet}: X^{\bullet} \otimes Y \otimes (\wedge^iV)^{\vee} \rightarrow X^{\bullet} \otimes Y \otimes (\wedge^{i+1}V)^{\vee}$ determined by:
\begin{align*}
 & \overline{D}_i^{\bullet}( f \otimes y \otimes \phi_{ v_{1} \wedge \ldots \wedge v_{n-i}}) \\
=& \sum_{j=1}^{n-i}(-1)^{j+1}(f\otimes \widetilde{v}_{j}.y \otimes \phi_{v_{1} \wedge \ldots \wedge \widehat{v}_{j} \wedge \ldots \wedge v_{n-i}}-\widetilde{v}_{j}.f\otimes y \otimes \phi_{v_{1} \wedge \ldots \wedge \widehat{v}_{j} \wedge \ldots \wedge v_{n-i}}),
\end{align*}

Similar to $\overline{D}_i$, the restriction of $\overline{D}_i^{\bullet}$ to $(X^{\bullet} \otimes Y \otimes (\wedge^i V)^{\vee})^W$ has image in $(X^{\bullet} \otimes Y \otimes (\wedge^{i+1} V)^{\vee})^W$. Denote by $D_i^{\bullet}$ the restriction of $\overline{D}_i^{\bullet}$ to $(X^{\bullet} \otimes Y \otimes (\wedge^i V)^{\vee})^W$. 

Define a linear isomorphism $\overline{\Psi}: X^* \otimes Y \otimes (\wedge^i V)^{\vee} \rightarrow \mathrm{Hom}_{\mathbb{C}}(X \otimes \wedge^i V, Y)$ as follows, where we also regard $X$ and $Y$ as vector spaces: for $f \otimes y \otimes \phi_{v_{1} \wedge \ldots \wedge v_{{n-i}}} \in X^* \otimes Y \otimes (\wedge^iV)^{\vee}$ has the action given by: for $f \otimes y \otimes \phi_{v_{1} \wedge \ldots \wedge v_{{n-i}}} \in X^* \otimes Y \otimes (\wedge^iV)^{\vee}$,
\[    \overline{\Psi}( f \otimes y \otimes \phi_{v_{1} \wedge \ldots \wedge v_{{n-i}}})(x \otimes u_{1} \wedge \ldots \wedge u_{i})=f(x)\phi_{v_{1} \wedge \ldots \wedge v_{{n-l}}}( u_1 \wedge \ldots \wedge u_i)y \in Y
\]
for $x \otimes u_1 \wedge \ldots \wedge u_i \in X \otimes \wedge^i V$. 
The map $\overline{\Psi}$ indeed depends on $i$, but we shall suppress the index $i$. By taking restriction on the space $(X^* \otimes Y \otimes (\wedge^i V)^{\vee}))^W$, we obtain the linear isomorphism:
\[ \Psi: (X^* \otimes Y \otimes (\wedge^i V)^{\vee})^W \rightarrow \mathrm{Hom}_{W}(X \otimes \wedge^i V, Y) .\]
Since $X^*$ and $X^{\bullet}$ can be naturally identified as vector spaces, the maps $\overline{\Psi}$ and $\Psi$ are also defined for the corresponding spaces involving $\bullet$ instead of $*$. 

Recall that $d_i^{\vee}$ and $\overline{d}_i^{\vee}$ are defined in Section \ref{sec complex ext}. 
\begin{lemma} \label{lem D d commute}
Let $X$ and $Y$ be $\mathbb{H}$-modules. Then
\begin{enumerate}
\item[(1)] For any $\omega \in X^* \otimes Y \otimes (\wedge^iV)^{\vee}$, $\overline{\Psi}(\overline{D}_i(\omega))=(-1)^{n-i+1}\overline{d}_{i,W}^{\vee}(\overline{\Psi}(\omega)) $.
\item[(2)] For any $\omega \in X^{\bullet} \otimes Y \otimes (\wedge^iV)^{\vee}$, $\overline{\Psi}(\overline{D}_i^{\bullet}(\omega))=(-1)^{n-i+1}\overline{d}_{i,W}^{\vee}(\overline{\Psi}(\omega)) $.
\end{enumerate}

\end{lemma}

\begin{proof}

Recall that $e_1, \ldots, e_n$ be the fixed basis for $V$. Let $\omega=f \otimes y \otimes \phi_{e_{k_1} \wedge \ldots \wedge e_{k_{n-i}}} \in X^* \otimes Y \otimes (\wedge^iV)^{\vee}$. By linearity, it suffices to check that 
\[\overline{\Psi}(\overline{D}_i(\omega))(x \otimes e_{k_1'}\wedge\ldots \wedge e_{k_{i+1}'})=(-1)^{n-i+1}\overline{d}_{i,W}^{\vee}(\overline{\Psi}(\omega))(x \otimes e_{k_1'}\wedge\ldots \wedge e_{k_{i+1}'}) \]
for any $x \in X$ and any indices $k_1', \ldots, k_{i+1}' \in \left\{1, \ldots, n \right\}$.

 Suppose $|\left\{ k_1 ,\ldots, k_{n-i} \right\} \cap \left\{ k_1', \ldots, k_{i+1}' \right\}|  \geq 2$. By Lemma \ref{lem equi inner wedge form},
\begin{align*}
& \overline{d}_{i}^{\vee}\overline{\Psi}(\omega)(x\otimes e_{k_1'} \wedge \ldots \wedge e_{k_{i+1}'}) =0= \overline{\Psi}(\overline{D}_i\omega)(x\otimes e_{k_1'} \wedge \ldots \wedge e_{k_{i+1}'})
\end{align*}

Suppose $|\left\{ k_1 ,\ldots, k_{n-i} \right\} \cap \left\{ k_1', \ldots, k_{i+1}' \right\}| =1$. Let $k_p$ and $k_q'$ be the unique pair of indices such that $e_{k_p}=e_{k'_q}$. 
Then 
\begin{align*}
& \overline{d}_{i,W}^{\vee}\overline{\Psi}(\omega)(x\otimes e_{k'_1} \wedge \ldots \wedge e_{k'_{i+1}}) \\
=& \overline{\Psi}(\omega)(d_{i,W}(x\otimes e_{k'_1} \wedge \ldots \wedge e_{k'_{i+1}})) \\
=& (-1)^{q+1}f(x)\phi_{e_{k_1} \wedge \ldots \wedge e_{k_{n-i}}}( e_{k_1'} \wedge \ldots \wedge \widehat{e}_{k_q'} \wedge \ldots \wedge e_{k'_{i+1}})\widetilde{e}_{k_q'}.y \\
 & \quad -  (-1)^{q+1}f(\widetilde{e}_{k_q'}.x)\phi_{e_{k_1} \wedge \ldots \wedge e_{k_{n-i}}}( e_{k_1'} \wedge \ldots \wedge \widehat{e}_{k_q'} \wedge \ldots \wedge e_{k'_{i+1}})y  \quad  \\
=& (-1)^{n-i-p} f(x)\phi_{e_{k_1} \wedge \ldots \wedge \widehat{e}_{k_p} \wedge \ldots \wedge e_{k_{n-i}}}(e_{k_1'} \wedge \ldots   \wedge e_{k_{ i+1}'})\widetilde{e}_{k_q'}.y  \\
 & \quad -  (-1)^{n-i-p} f(\widetilde{e}_{k_q'}.x)\phi_{e_{k_1} \wedge \ldots \wedge \widehat{e}_{k_p} \wedge \ldots \wedge e_{k_{n-i}}}(e_{k_1'} \wedge \ldots   \wedge e_{k_{ i+1}'})y \quad \mbox{ (by Lemma \ref{lem equi inner wedge form})} \\
=& (-1)^{n-i-p} f(x)\phi_{e_{k_1} \wedge \ldots \wedge \widehat{e}_{k_p} \wedge \ldots \wedge e_{k_{n-i}}}(e_{k_1'} \wedge \ldots   \wedge e_{k_{ i+1}'})\widetilde{e}_{k_p}.y  \\
 & \quad -  (-1)^{n-i-p} f(\widetilde{e}_{k_p}.x)\phi_{e_{k_1} \wedge \ldots \wedge \widehat{e}_{k_p} \wedge \ldots \wedge e_{k_{n-i}}}(e_{k_1'} \wedge \ldots   \wedge e_{k_{ i+1}'})y \quad \mbox{ (by $e_{k_p}=e_{k_q'}$)}  \\
=& (-1)^{n-i-p} f(x)\phi_{e_{k_1} \wedge \ldots \wedge \widehat{e}_{k_p} \wedge \ldots \wedge e_{k_{n-i}}}(e_{k_1'} \wedge \ldots   \wedge e_{k_{ i+1}'})\widetilde{e}_{k_p}.y  \\
 & \quad + (-1)^{n-i-p-1} (\widetilde{e}_{k_p}.f)(x)\phi_{e_{k_1} \wedge \ldots \wedge \widehat{e}_{k_p} \wedge \ldots \wedge e_{k_{n-i}}}(e_{k_1'} \wedge \ldots   \wedge e_{k_{ i+1}'})y     \mbox{ (by Lemma \ref{lem star form})}\\
=& (-1)^{n-i+1}\overline{\Psi}(D_i( f \otimes y \otimes \phi_{e_{k_1} \wedge \ldots \wedge e_{k_{n-i}}}))(x\otimes e_{k'_1} \wedge \ldots \wedge e_{k'_{i+1}}) \quad \\
=& (-1)^{n-i+1} \overline{\Psi}(\overline{D}_i(\omega))(x\otimes e_{k'_1} \wedge \ldots \wedge e_{k'_{i+1}}) 
\end{align*}
This completes the proof for (1).

The proof for (2) follows the same style of computations. (One of the difference is in the fifth equality of the computation in the second case and that explains why the definition of $\overline{D}^{\bullet}_i$ and $\overline{D}^*_i$ differs by a sign in a term.) 
\end{proof}

\begin{lemma} \label{lem complexes isomorphic}
We have the following:
\begin{enumerate}
\item[(1)] $D^2=0$ and $(D^{\bullet})^2=0$,
\item[(2)] The complex $\mathrm{Hom}_{W}(\mathrm{Res}_WX \otimes \wedge^iV, Y)$ with differentials $d_i^*$ is naturally isomorphic to the complex $(X^* \otimes Y \otimes (\wedge^i V)^{\vee})^W$ with differentials $D_i$. 
\item[(3)] The complex $\mathrm{Hom}_{W}(\mathrm{Res}_WX \otimes \wedge^iV, Y)$ with differentials $d_i^*$ is naturally isomorphic to the complex $(X^{\bullet} \otimes Y \otimes (\wedge^i V)^{\vee})^W$ with differentials $D_i^{\bullet}$.  
\end{enumerate}

\end{lemma}

\begin{proof}
By Lemma \ref{lem D d commute}(1) and the fact that $\Psi$ is an isomorphism, $D_i=\Psi^{-1} \circ d_i^* \circ \Psi $. Then (1) follows from Lemma \ref{lem well defined d}. (2) follows from Lemma \ref{lem D d commute} (1). The proof for (3) and another assertion about $D^{\bullet}$ in (1)  is similar.
\end{proof}

\begin{proposition} \label{prop equiv ext group}
Let $X, Y$ be finite dimensional $\mathbb{H}$-modules. We have:
\begin{enumerate}
\item $\mathrm{Ext}^i_{\mathbb{H}}(X, Y^*) \cong\mathrm{Ext}^i_{\mathbb{H}}(Y, X^*)$;
\item $\mathrm{Ext}^i_{\mathbb{H}}(X, Y^{\bullet}) \cong \mathrm{Ext}^i_{\mathbb{H}}(Y, X^{\bullet})$;
\item $\mathrm{Ext}^i_{\mathbb{H}}(X, \theta(Y)) \cong \mathrm{Ext}^i_{\mathbb{H}}(\theta(X), Y)$.
\end{enumerate}
Those linear isomorphisms between them are natural.

\end{proposition}

\begin{proof}
We first prove that $\mathrm{Ext}^i_{\mathbb{H}}(X, Y^*)=\mathrm{Ext}^i_{\mathbb{H}}(Y, X^*)$. By Lemma \ref{lem complexes isomorphic} (2),
\[\mathrm{Ext}^i_{\mathbb{H}}(X, Y^*) \cong \ker D_{(X^* \otimes Y^* \otimes (\wedge^{i}V)^{\vee})^W}/\im D_{ (X^* \otimes Y^* \otimes (\wedge^{i-1}V)^{\vee})^W}, \]
and
\[\mathrm{Ext}^i_{\mathbb{H}}(Y, X^*) \cong \ker D_{ (Y^* \otimes X^* \otimes (\wedge^{i}V)^{\vee})^W}/\im D_{(Y^* \otimes X^* \otimes (\wedge^{i-1}V)^{\vee})^W}. \]
There is a natural isomorphism between the spaces $(X^* \otimes Y^* \otimes (\wedge^{n-i}V)^{\vee})^W$ and $ (Y^* \otimes X^* \otimes (\wedge^{n-i}V)^{\vee})^W$. It is straightforward to verify the isomorphism induces an isomorphism between the corresponding complexes by using (\ref{eqn D_i def}). The proof for (2) is similar. 

For (3), let $P^{i} \rightarrow X$ be a projective resolution of $X$. Then $\theta(P^i)$ is still a projective object and $\theta(P^i) \rightarrow \theta(X)$ is a projective resolution of $\theta(X)$. There is a natural isomorphism $\mathrm{Hom}_{\mathbb{H}}(\theta(P^i), Y) \cong \mathrm{Hom}_{\mathbb{H}}(P^i, \theta(Y))$. Hence $\mathrm{Ext}^i_{\mathbb{H}}(\theta(X), Y)=\mathrm{Ext}^i_{\mathbb{H}}(X, \theta(Y))$.
\end{proof}



\subsection{Iwahori-Matsumoto dual} \label{ss IM invo}

\begin{definition} \label{def IM daul}
The Iwahori-Matsumoto involution $\iota$ is an automorphism on $\mathbb{H}$ determined by 
\[ \iota(v)=-v \quad \mbox{ for $v \in V$}, \iota(w)=\sgn(w) w  \quad \mbox{ for $w \in W$ }.
\]
This defines a map, still denoted $\iota$, from the set of $\mathbb{H}$-modules to the set of $\mathbb{H}$-modules.
\end{definition}
\begin{lemma} \label{lem iota v}
For any $v \in V$, $\iota(\widetilde{v})=-\widetilde{v}$. 
\end{lemma}
\begin{proof}
This follows from $\iota(s_{\alpha})=-s_{\alpha}$ and definitions.
\end{proof}

\begin{lemma} \label{lem property bullet form}
Let $Y$ be an $\mathbb{H}$-module. Let $\kappa$ be the natural (vector space) bijection from $Y$ to $\iota(Y)$ so that $h.\kappa(y)=\kappa(\iota(h).y)$. Define a bilinear pairing $\langle , \rangle_Y^{\bullet}:   Y  \times \iota(Y)^{\bullet} \rightarrow \mathbb{C}$ by $\langle y, g \rangle_Y^{\bullet}=g(\kappa(y))$. Then
\begin{enumerate}
\item[(1)] for $v \in V$, $\langle \widetilde{v}.y, g \rangle^{\bullet}_Y = \langle y, -\widetilde{v}.g \rangle^{\bullet}_Y$,  
\item[(2)] for $w \in W$, $\langle t_w.y, g \rangle^{\bullet}_Y= \sgn(w)\langle y, t_w^{-1}.g \rangle_Y^{\bullet}$,
\item[(3)] $\langle , \rangle^{\bullet}_Y$ is non-degenerate.
\end{enumerate}
\end{lemma}
\begin{proof}
By a direct computation, $\widetilde{v}^{\bullet}=\widetilde{v}$ and then (1) follows from Lemma \ref{lem iota v} and the definitions. (2) and (3) follow from the definitions. 
\end{proof}

\begin{proposition} \label{prop ext im invo}
For $\mathbb{H}$-modules $X$ and $Y$, $\mathrm{Ext}^i_{\mathbb{H}}(X, \iota(Y))\cong \mathrm{Ext}^i_{\mathbb{H}}(\iota(X), Y)$.
\end{proposition}

\begin{proof}
 Let $P^{i} \rightarrow X$ be a projective resolution of $X$. Then $\iota(P^i)$ is still a projective object and $\iota(P^i) \rightarrow \iota(X)$ is a projective resolution of $\iota(X)$. There is a natural isomorphism $\mathrm{Hom}_{\mathbb{H}}(\iota(P^i), Y) \cong \mathrm{Hom}_{\mathbb{H}}(P^i, \iota(Y))$. Hence $\mathrm{Ext}^i_{\mathbb{H}}(\iota(X), Y)\cong \mathrm{Ext}^i_{\mathbb{H}}(X, \iota(Y))$.
\end{proof}

\subsection{Duality theorem}

In this section, we state and prove a duality on $\mathrm{Ext}_{\mathbb{H}}$-groups.

\begin{theorem} \label{thm poin dua} Let $\mathbb{H}$ be the graded affine Hecke algebra associated to a root datum $( R, V, R^{\vee}, V^{\vee}, \Pi)$ and a parameter function $k: \Pi \rightarrow \mathbb{C}$ (Definition \ref{def graded affine}). Let $n=\dim V$. Let $X$ and $Y$ be finite dimensional $\mathbb{H}$-modules. Let $X^*$ be the $*$-dual of $X$ in Definition \ref{def star bullet duals}. Let $\iota(Y)$ be the Iwahori-Matsumoto dual in Definition \ref{def IM daul} and let $\iota(Y)^{\bullet}$ be the $\bullet$-dual of $\iota(Y)$ in Definition \ref{def star bullet duals}. Then there exists a natural non-degenerate pairing 
\[  \mathrm{Ext}^i_{\mathbb{H}}(X, Y) \times \mathrm{Ext}_{\mathbb{H}}^{n-i}(X^*, \iota(Y)^{\bullet})  \rightarrow \mathbb{C}.
\]
\end{theorem}
\begin{proof}
We divide the proof into few steps. \\
\noindent
{\bf Step 1: Construct non-degenerate bilinear pairings.} 

The space
\[  \Hom_{W}(X \otimes \wedge^i V, Y) \times  \Hom_{W}(X^* \otimes \wedge^{n-i} V, \iota(Y)^{\bullet})
\]
is identified with 
\[ (X^* \otimes Y\otimes (\wedge^i V)^{\vee})^W \times (X  \otimes \iota(Y)^{\bullet}\otimes (\wedge^{n-i} V)^{\vee})^W 
\]
as in Section \ref{sec complex duals}. Let $\langle , \rangle_X^*$ be the bilinear pairing on $ X^* \times X$ such that $\langle f, x \rangle_X^*=f(x)$ for $f \in X^*$ and $x \in X$. Let $\langle , \rangle_Y^{\bullet}$ be the bilinear pairing on $  Y \times \iota(Y)^{\bullet}$ such that $\langle y, g \rangle_Y^{\bullet} =g(\kappa(y))$ for $g \in \iota(Y)^{\bullet}$ and $y \in Y$. Here $\kappa$ is defined as in Lemma \ref{lem property bullet form}. 

For each $i$, we first define the pairing $\langle , \rangle_{X, Y, \wedge^iV}$ on a larger space $(X^* \otimes Y\otimes (\wedge^i V)^{\vee}) \times (X \otimes \iota(Y)^{\bullet} \otimes (\wedge^{n-i} V)^{\vee})$ via the product of the pairings $\langle , \rangle_X^*$, $\langle , \rangle_Y^{\bullet}$ and $\langle , \rangle_{(\wedge^i V)^{\vee}}$ i.e.
\[ \langle f \otimes y\otimes \phi_{v_1 \wedge \ldots \wedge v_{n-i}} , x   \otimes g \otimes \phi_{v_{n-i+1} \wedge \ldots \wedge v_n}\rangle_{X, Y. \wedge^iV} =\langle f, x \rangle_X^* \langle y, g \rangle_Y^{\bullet} \langle \phi_{v_1 \wedge \ldots \wedge v_{n-i}}, \phi_{v_{n-i+1} \wedge \ldots \wedge v_n} \rangle_{(\wedge^iV)^{\vee}} .
\]
Since all the pairings $\langle, \rangle_X^*$, $\langle, \rangle_Y^{\bullet}$ and $\langle ,\rangle_{(\wedge^iV)^{\vee}}$ are bilinear, $\langle , \rangle_{X, Y, \wedge^iV}$ is bilinear and well-defined.

 This pairing $\langle , \rangle_{X, Y, \wedge^iV}$ is non-degenerate because $\langle, \rangle_X^*$, $\langle, \rangle_Y^{\bullet}$ and $\langle ,\rangle_{(\wedge^iV)^{\vee}}$ are non-degenerate. Note that $\langle , \rangle_{X, Y, \wedge^iV}$ is $W$-invariant, which follows from Lemma \ref{lem star form}(2), Lemma \ref{lem property bullet form}(2) and Lemma \ref{lem property wedge form}. 

In order to see $\langle , \rangle_{X, Y, \wedge^iV}$ restricted on $(X^*\otimes Y\otimes (\wedge^i V)^{\vee} )^W \times (X \otimes \iota(Y)^{\bullet}\otimes (\wedge^{n-i} V)^{\vee} )^W$ is still non-degenerate, we pick $\omega \in (X^*\otimes Y\otimes (\wedge^i V)^{\vee} )^W$. There exists $\omega' \in X \otimes (\wedge^{n-i} V)^{\vee} \otimes \iota(Y)^{\bullet} \otimes (\wedge^{n-i} V)^{\vee}$ such that $\langle \omega , \omega' \rangle_{X, Y, \wedge^i V} \neq 0$. Then by the $W$-invariance of $\langle , \rangle_{X, Y, \wedge^i V}$, we have $\langle \omega , \sum_{w \in W} w(\omega') \rangle_{X, Y, \wedge^i V} \neq 0$, as desired. Hence 
$\langle , \rangle_{X, Y, \wedge^iV}$ restricted on $(X^* \otimes Y\otimes (\wedge^i V)^{\vee})^W \times (X \otimes \iota(Y^{\bullet})\otimes (\wedge^{n-i} V)^{\vee})^W$  is still non-degenerate. \\

\noindent
{\bf Step 2: Compute the adjoint operator of $\overline{D}$ for $\langle , \rangle_{X, Y, \wedge^i V}$}

Recall that $\overline{D}$ is defined in (\ref{eqn D_i def}). For notational simplicity, set $\overline{D}^1_i=\overline{D}_{ (X^* \otimes Y \otimes (\wedge^iV)^{\vee})}$ and $\overline{D}^2_{n-i-1}=\overline{D}_{(X \otimes \iota(Y)^{\bullet} \otimes (\wedge^{n-i-1}V)^{\vee})}$, where we regard $X=(X^*)^*$. Recall that we fixed a basis $e_1, \ldots, e_n$ for $V$. We first show that
\begin{align} \nonumber
& \langle \overline{D}_i^1(f \otimes y \otimes \phi_{e_{k_1}\wedge \ldots \wedge e_{k_{n-i}}}),  x \otimes g \otimes \phi_{e_{k'_1} \wedge\ldots \wedge e_{k'_{i+1}}} \rangle_{X,Y,\wedge^{i+1}V} \\
=& \pm \langle f \otimes y \otimes \phi_{e_{k_1}\wedge \ldots \wedge e_{k_{n-i}}},  \overline{D}_{n-i-1}^2(x \otimes g \otimes \phi_{e_{k'_1} \wedge\ldots \wedge e_{k'_{i+1}}}) \rangle_{X, Y, \wedge^{i} V} \label{eqn adjoint D}
\end{align}
( $f \in X^*$, $x \in X$, $g \in Y^*$, $y \in Y$, $\phi_{e_{k_1}\wedge \ldots \wedge e_{k_{n-i}}} \in (\wedge^i V)^{\vee}$ and  $\phi_{e_{k'_1} \wedge\ldots \wedge e_{k'_{i+1}}} \in (\wedge^{n-i-1} V)^{\vee}$).

We divide into two cases. Suppose $|\left\{ k_1 ,\ldots, k_{n-i} \right\} \cap \left\{ k_1', \ldots, k_{i+1}' \right\}| \geq 2$. Then,
\begin{align*}
 & \langle \overline{D}_i^1(f \otimes y \otimes \phi_{e_{k_1}\wedge \ldots \wedge e_{k_{n-i}}}),  x \otimes g \otimes \phi_{e_{k'_1} \wedge\ldots \wedge e_{k'_{i+1}}} \rangle_{X, Y, \wedge^{i+1} V}\\
=& 0 =(-1)^{n-i}\langle f \otimes y \otimes \phi_{e_{k_1}\wedge \ldots \wedge e_{k_{n-i}}},  \overline{D}_{n-i-1}^2(x \otimes g \otimes \phi_{e_{k'_1} \wedge\ldots \wedge e_{k'_{i+1}}})) \rangle_{X, Y, \wedge^{i} V}
\end{align*}

For the second case, suppose $|\left\{ k_1 ,\ldots, k_{n-i} \right\} \cap \left\{ k_1', \ldots, k_{i+1}' \right\}| =1$. Let $k_p$ and $k_q'$ be the unique pair of indices such that $e_{k_p}=e_{k'_q}$. Then 
\begin{align*}
 & \langle \overline{D}_i^1(f \otimes y \otimes \phi_{e_{k_1}\wedge \ldots \wedge e_{k_{n-i}}}),  x \otimes g \otimes \phi_{e_{k'_1} \wedge\ldots \wedge e_{k'_{i+1}}} \rangle_{X, Y, \wedge^{i+1} V}\\
=& (-1)^{p+1} \langle \widetilde{e}_{k_p}.f , x \rangle_X^*\ \langle  y, g \rangle_Y^{\bullet}\ \langle \phi_{e_{k_1} \wedge \ldots \wedge \widehat{e}_{k_p} \wedge \ldots \wedge e_{k_{n-i}}}, \phi_{e_{k_{1}'} \wedge \ldots \wedge e_{k_{i+1}'}} \rangle_{(\wedge^{n-i-1}V)^{\vee}}  \\
  & \quad + (-1)^{p+1}\langle f , x \rangle_X^*\ \langle \widetilde{e}_{k_p}.y, g   \rangle_Y^{\bullet}\ \langle  \phi_{e_{k_1} \wedge \ldots \wedge \widehat{e}_{k_p} \wedge \ldots \wedge e_{k_{n-i}}}, \phi_{e_{k_{1}'} \wedge \ldots \wedge e_{k_{i+1}'}} \rangle_{(\wedge^{n-i-1}V)^{\vee}} \\
=& (-1)^{p+1}\langle f , -\widetilde{e}_{k_p}.x \rangle_X^*\ \langle y, g \rangle_Y^{\bullet}\  \langle \phi_{e_{k_1} \wedge \ldots \wedge \widehat{e}_{k_p} \wedge \ldots \wedge e_{k_{n-i}}}, \phi_{e_{k_{1}'} \wedge \ldots \wedge e_{k_{i+1}'}} \rangle_{(\wedge^{n-i-1}V)^{\vee}} \\
 & \quad +(-1)^{p+1}\langle f , x \rangle_X^* \langle y, -\widetilde{e}_{k_p}.g \rangle_Y^{\bullet}\ \langle  \phi_{e_{k_1} \wedge \ldots \wedge \widehat{e}_{k_p} \wedge \ldots \wedge e_{k_{n-i}}}, \phi_{e_{k_{1}'} \wedge \ldots \wedge e_{k_{i+1}'}} \rangle_{(\wedge^{n-i-1}V)^{\vee}} \\
=& (-1)^{n-i+q}\langle f , -\widetilde{e}_{k_q'}.x \rangle_X^*\  \langle  y, g \rangle_Y^{\bullet}\ \langle \phi_{e_{k_1} \wedge \ldots \wedge e_{k_{n-i}}},   \phi_{e_{k_1'} \wedge \ldots \wedge \widehat{e}_{k'_q} \wedge \ldots  \wedge e_{k_{i+1}}} \rangle_{(\wedge^{n-i}V)^{\vee}} \\
 & \quad +(-1)^{n-i+q} \langle f , x \rangle_X^*\ \langle  y, -\widetilde{e}_{k_q'}.g \rangle_Y^{\bullet}\ \langle \phi_{e_{k_1} \wedge \ldots \wedge e_{k_{n-i}}},  \phi_{e_{k_1'} \wedge \ldots \wedge \widehat{e}_{k'_q} \wedge \ldots  \wedge e_{k_{i+1}}} \rangle_{(\wedge^{n-i}V)^{\vee}}  \\
=&(-1)^{n-i}\langle f \otimes y \otimes \phi_{e_{k_1}\wedge \ldots \wedge e_{k_{n-i}}},  \overline{D}_{n-i-1}^2(x \otimes g \otimes \phi_{e_{k'_1} \wedge\ldots \wedge e_{k'_{i+1}}}) \rangle_{X, Y, \wedge^{i} V}
\end{align*}
 The second equality follows from Lemma \ref{lem star form}(1) and Lemma \ref{lem property bullet form}(1). The third equality follows from $e_{k_p}=e_{k'_q}$. Hence we have shown the equation (\ref{eqn adjoint D}). By linearity, we have for $\omega_1 \in (X^* \otimes Y \otimes (\wedge^i V)^{\vee})^W$ and $\omega_2 \in (X \otimes Y \otimes (\wedge^{n-i-1}V)^{\vee})^W$,
\begin{align} \label{eqn adjoint PD}
\langle D_i^1 \omega_1, \omega_2 \rangle_{X, Y, \wedge^{i+1}V}  &=(-1)^{n-i}  \langle  \omega_1, D_{n-i-1}^2\omega_2 \rangle_{X, Y, \wedge^{i}V},
\end{align}
where $D_i^1=D_{ (X^* \otimes Y \otimes (\wedge^iV)^{\vee})^W}$ and $D_{n-i-1}^2=D_{(X \otimes \iota(Y)^{\bullet} \otimes (\wedge^{n-i-1}V)^{\vee})^W}$.






Then by (\ref{eqn adjoint PD}) with some linear algebra argument, the pairing $\langle .,.\rangle_{X, Y. \wedge^iV}$ descends to 
\[\ker D_i^1 /\im D_{i-1}^1 \times \ker D_{n-i}^2 / \im D_{n-i-1}^2 \rightarrow \mathbb{C}.\]
By Proposition \ref{prop complex ext gp} and Lemma \ref{lem complexes isomorphic}(2), we have a natural non-degenerate pairing on 
\[  \mathrm{Ext}^i_{\mathbb{H}}(X, Y) \times \mathrm{Ext}_{\mathbb{H}}^{n-i}(X^*, \iota(Y)^{\bullet})  \rightarrow \mathbb{C}. \qedhere
\]
\end{proof}

\begin{remark} \label{rmk comment dua result}
We give few comments concerning the statement and the proof of Theorem \ref{thm poin dua}.
\begin{enumerate}
\item[(1)] If $X$ and $Y$ have the same central character, then $X^*$ and $\iota(Y^{\bullet})$ also have the same central character.
\item[(2)]  The use of the element $\widetilde{v}$ makes the computation in step 2 of the proof easier. 
\item[(3)] The choice of the duals is necessary to have a nice adjoint operator of $\overline{D}$ for the pairing $\langle .,.\rangle_{X, Y, \wedge^iV}$ in step 2. By Proposition \ref{prop equiv ext group}, one also obtains a non-degenerate pairing
\[  \mathrm{Ext}^i_{\mathbb{H}}(X, Y) \times \mathrm{Ext}_{\mathbb{H}}^{n-i}(X^{\bullet}, \iota(Y)^{*})  \rightarrow \mathbb{C}.
\]
\item[(4)] The Iwahori-Matsumoto involution is necessary to show the pairing $\langle .,.\rangle_{X, Y, \wedge^iV}$ is $W$-invariant and so non-degenerate in step 1.
\end{enumerate}
\end{remark}

\begin{remark}
If we replace the anti-involutions $*$ and $\bullet$ with Hermitian anti-involutions studied in \cite{BC}, one can still obtain an analogous statement for Theorem \ref{thm poin dua} involving those Hermitian anti-involutions.
\end{remark}

\subsection{Some consequences of the duality}

The duality sometimes provides an easier way to compute $\mathrm{Ext}$-groups. The first immediate application is to compute the top $\mathrm{Ext}$-group, provided that the $\mathbb{D}$-dual is known. Here the operator $\mathbb{D}$ is defined below:
\begin{definition} \label{def dualizing module}
For a finite-dimensional $\mathbb{H}$-module $X$, define $\mathbb{D}(X)= \iota(X^{\bullet})^* $. By Lemma \ref{lem two dual theta}, $\mathbb{D}(X) \cong \iota(\theta(X))$.
\end{definition}

\begin{corollary} \label{cor duality}
Let $X$ and $Y$ be finite-dimensional irreducible $\mathbb{H}$-modules. Then
\[  \mathrm{Ext}^n_{\mathbb{H}}(X,Y)  \cong \left\{ \begin{array}{c c} \mathbb{C} & \mbox{ if $X \cong \mathbb{D}(Y)$ } \\   0 & \mbox{ if $X \not\cong \mathbb{D}(Y)$ } \end{array} \right. .
\]
\end{corollary}

\begin{proof}
This follows from Theorem \ref{thm poin dua} and Schur's Lemma. 
\end{proof}

One may also expect that the symmetry from the duality can help compute some $\mathrm{Ext}$-groups. We shall see in Section \ref{s ext ds} it makes the computation between a tempered module and a discrete series simpler. Other than that, for $R$ of rank $2$, the $\mathrm{Ext}$-groups between most of irreducible modules can be computed with only knowing the $W$-structure and central characters by Proposition \ref{prop complex ext gp} and Corollary \ref{cor duality}. For some higher rank cases, some extra information such as composition factors of standard modules may be needed for computations. Further discussions will appear in \cite{Ch2}.

\section{Duality for the Yoneda product} \label{s dua YP}

In this section, we prove the duality for Theorem \ref{thm poin dua} in the level of the Yoneda product. This information is more useful for studying the extension algebra of graded affine Hecke algebra modules. The proof for the improved duality is basically extending the one of Theorem \ref{thm poin dua}. In particular, we have to put the bilinear forms in the proof of Theorem \ref{thm poin dua} in a suitable form, that is the content of Section \ref{ss bilinear pair} below.




\subsection{Yoneda extension}

We first review the Yoneda extension. 

We first recall the interpretation of $\mathrm{Ext}$-groups in terms of the Yoneda $\mathrm{Ext}$-groups. Let $\mathcal S(X, Y)$ be  the collection of exact sequences of the form:
\[  0 \rightarrow Y \rightarrow Z_i \rightarrow \ldots \rightarrow Z_1 \rightarrow X \rightarrow 0, \]
where $Z_k$ are $\mathbb{H}$-modules. Denote $\mathrm{YExt}^i_{\mathbb{H}}(X, Y)$ by the equivalence classes of $\mathcal S(X,Y)$ under the equivalence relation generated by the relation that for $S, S' \in \mathcal S(X, Y)$, $S \sim S'$ if and only if there exists a commutative diagram:
\[\xymatrix{ 
S: &\ar[r]0 & Y \ar[d]^{id} \ar[r] & Z_i \ar[d]\ar[r]& \cdots \cdots \ar[r]  & Z_1 \ar[d] \ar[r]& X \ar[d]^{id} \ar[r] &  0 \\
S': & 0   \ar[r] & Y \ar[r] & Z_i'   \ar[r] & \cdots \cdots  \ar[r]  & Z_1'  \ar[r]& X \ar[r] & 0 }
\]

For any $S\in \mathcal S(X, Y)$, we have a chain morphism between two complexes
\[\xymatrix{ 
     &\ar[r]0 & \ker d_{i-2} \ar[d]^{\eta} \ar[r] & P_{i-1} \ar[d]\ar[r]^{d_{i-2}} & \cdots \cdots \ar[r]  & P_0 \ar[d] \ar[r]& X \ar[d]^{id} \ar[r] &  0 \\
S: & 0   \ar[r] & Y \ar[r] & Z_i   \ar[r] & \cdots \cdots  \ar[r]  & Z_1  \ar[r]& X \ar[r] & 0 },
\]
where $P_j$ form a projective resolution of $X$. This gives rise a connecting homomorphism $\partial$ and the following exact sequence:
\[    \mathrm{Hom}_{\mathbb{H}}(P_{i-1}, Y) \rightarrow \mathrm{Hom}_{\mathbb{H}}(\ker d_{i-2}, Y) \stackrel{\partial}{\rightarrow} \mathrm{Ext}^i_{\mathbb{H}}(X, Y) \rightarrow 0.
\]
We obtain an element $\partial(\eta) \in \mathrm{Ext}^i_{\mathbb{H}}(X, Y)$. This defines a map from $\mathcal S(X, Y)$ to $\mathrm{Ext}^i_{\mathbb{H}}(X, Y)$. Since two elements in the same equivalence class of $\mathcal S(X, Y)$ have the same image under the map, we also have a map from $\mathrm{YExt}_{\mathbb{H}}^i(X, Y)$ to $\mathrm{Ext}_{\mathbb{H}}^i(X, Y)$, which is indeed an isomorphism (see e.g. \cite[Chapter III.5]{Ma} or \cite[Chapter 3.4]{We}). Denote the inverse of the isomorphism by $\mathrm{Yon}_{X, Y}^i: \mathrm{Ext}^i_{\mathbb{H}}(X, Y) \rightarrow \mathrm{YExt}^i_{\mathbb{H}}(X, Y)$. We may sometimes write $\mathrm{Yon}$ for $\mathrm{Yon}_{X,Y}^i$ for simplicity.


\subsection{Alternate form of bilinear pairings} \label{ss bilinear pair}

We need several notations. We fix a basis $e_1, \ldots, e_n$ for $V$ as in Section \ref{ss pairing wedge V}. Let $n=\dim V$. Let $X$ be a finite-dimensional $\mathbb{H}$-module. Set $P_i= \mathbb{H} \otimes_{\mathbb{C}[W]}(\mathrm{Res}_WX \otimes \wedge^iV)$ and set ${}^{\iota}P_{i}=\mathbb{H}\otimes_{\mathbb{C}[W]}(\mathrm{Res}_W \iota(X)^{\bullet} \otimes \wedge^{i} V)$, which are projective modules. We use $d_i$ for the differential map $d_{i,X}$ and use ${}^{\iota}d_i$ for the differential map ${}^{\iota}d_{i,\iota(X)^{\bullet}}$ (see Section \ref{sec kos resol} for  $d_{i,X}$  and ${}^{\iota}d_{i,\iota(X)^{\bullet}}$). The maps $d_i$ and ${}^{\iota}d_i$ induce dual maps $d_i^*: P_i^* \rightarrow P_{i+1}^*$ and ${}^{\iota}d_i^* : {}^{\iota}P_i^* \rightarrow {}^{\iota}P_{i+1}^*$ respectively. 

Let $\vartheta_i: \mathrm{Res}_W X \otimes \wedge^{i}V \rightarrow P_i$  be the natural injective map such that the image of $\vartheta_{i}$ is the subspace $1 \otimes (\mathrm{Res}_W X \otimes \wedge^{i}V)$ of ${}^{\iota}P_{n-i}$.  Similarly, let ${}^{\iota}\vartheta_{i}: \mathrm{Res}_W\iota(X)^{\bullet} \otimes \wedge^{i}V \rightarrow {}^{\iota}P_{i}$ be the natural inclusion such that the image of ${}^{\iota}\vartheta_{i}$ is the subspace $1 \otimes (\mathrm{Res}_W\iota(X)^{\bullet} \otimes \wedge^{i}V)$ of ${}^{\iota}P_{i}$. Let $(\mathrm{Res}_W \iota(X)^{\bullet} \otimes \wedge^{i} V)^{\vee}$ be the dual $W$-representation of $\mathrm{Res}_W\iota(X)^{\bullet} \otimes \wedge^{i}V$. ${}^{\iota}\vartheta_{i}$ induces a dual map ${}^{\iota}\vartheta_{i}^{*}:({}^{\iota}P_{i})^{*}\rightarrow   (\mathrm{Res}_W \iota(X)^{\bullet} \otimes \wedge^{i} V)^{\vee}  $.

 Let $B_{i}$ be the linear isomorphism from  $ \mathrm{Res}_WX \otimes \wedge^{i}V \rightarrow (\mathrm{Res}_W \iota(X)^{\bullet} \otimes \wedge^{n-i} V)^{\vee}$ such that $B_{i}(x \otimes v_{1} \wedge \ldots \wedge v_{i})(f \otimes v_{i+1} \wedge \ldots \wedge v_n)=f(x) \langle v_1 \wedge \ldots \wedge v_i, v_{i+1} \wedge \ldots \wedge v_{n} \rangle_{\wedge^{i}V}$, where $\langle , \rangle_{\wedge^{i}V}$ is defined as in Section \ref{ss pairing wedge V}.  For an $\mathbb{H}$-map $\alpha: P_{i} \rightarrow ({}^{\iota}P_{n-i})^{*}$, we define $\mathrm{Tr}_{i}(\alpha)=\mathrm{trace}(B_{i}^{-1} \circ {}^{\iota}\vartheta_{n-i}^{*} \circ\alpha_i \circ \vartheta_{i})$ ({\it c.f.} the bilinear pairings in Lemma \ref{lem star form} and Section \ref{ss pairing wedge V}), where $\mathrm{trace}$ is the usual trace of a linear endomorphism.

Let $Y$ be a finite-dimensional $\mathbb{H}$-module. For notation simplicity, we use $d_i^{\vee}$ for the differential map $d_{i,X}^{\vee, Y}$ and ${}^{\iota}d_i^{\vee}$ for $d_{i,\iota(X)^{\bullet}}^{\vee, Y^*}$ in the remainder of this subsection (see Section \ref{sec complex ext} for the notion of $d_{i,X}^{\vee, Y}$ and $d_{i,\iota(X)^{\bullet}}^{\vee, Y^*}$). We also write $d_{i,W}^{\vee}$ for $d_{i,X,W}^{\vee,Y}$ and ${}^{\iota}d_{i,W}^{\vee}$ for $d_{i, \iota(X^{\bullet}),W}^{\vee, Y^*}$. The map ${}^{\iota}d_i$ also induces a natural map, denoted ${}^{\iota}d_i^{\vee,*}$, from $ \mathrm{Hom}_{\mathbb{H}}(Y, {}^{\iota}P_i^{*})$ to $ \mathrm{Hom}_{\mathbb{H}}(Y, {}^{\iota}P_{i+1}^*)$. We also have induced maps of ${}^{\iota}d_{i,W}^{\vee,*}$ in the level of $W$, but we shall describe more explicitly below.

We have the following diagram:
\[\xymatrix{ 
 \mathrm{Hom}_{\mathbb{H}}({}^{\iota}P_i, Y^{*}) \ar[r]^{\sim} \ar[d]^{\mathrm{Frobenius\ reciprocity}} &   \mathrm{Hom}_{\mathbb{H}}(Y, {}^{\iota}P_i^{*}) \ar[d]^{T}  \\
 \mathrm{Hom}_{W}(\mathrm{Res}_W\iota(X)^{\bullet} \otimes \wedge^iV, \mathrm{Res}_WY^*) \ar[r]^{\sim} &  \mathrm{Hom}_{W}( \mathrm{Res}_WY , (\mathrm{Res}_W\iota(X)^{\bullet} \otimes \wedge^iV)^{\vee})  },
\]
where the top and bottom maps are the natural isomorphisms, and the map $T$ can be expressed as
\begin{align} \label{eqn dual Frobenius}
T(\alpha)= {}^{\iota}\vartheta_{i}^{*} \circ \alpha .
\end{align} For instance, the top isomorphism denoted by $K$ is given by: for $\eta \in \mathrm{Hom}_{\mathbb{H}}({}^{\iota}P_i, Y^{*})$, $(K(\eta)(y))(h \otimes f \otimes v_1 \wedge \ldots \wedge v_i)=\eta(h \otimes f \otimes v_1 \wedge \ldots \wedge v_i)(y)$. We shall also write $\eta^*$ for $K(\eta)$. It is straightforward to check the diagram is commutative and so $T$ is also an isomorphism.

Thus we now have new complexes and we again translate the differential maps to the new complexes. The map ${}^{\iota}d_i^{\vee, *}: \mathrm{Hom}_{\mathbb{H}}(Y, {}^{\iota}P_{i+1}^*) \rightarrow \mathrm{Hom}_{\mathbb{H}}(Y, {}^{\iota}P_{i}^*)$ can be described as, where $\beta \in \mathrm{Hom}_{\mathbb{H}}(Y, {}^{\iota}P_{i+1}^*)$, $y \in Y$, $f \in \iota(X)^{\bullet}$ and $v_1, \ldots, v_{i+1} \in V$:
\begin{align*}
 & ({}^{\iota}d_i^{\vee, *}(\beta)(y))(h \otimes f \otimes v_1 \wedge \ldots \wedge v_{i+1}) \\
=& \sum_{j=1}^{i+1} (-1)^{j+1}\beta(\widetilde{v}_j^{*}.y)(h \otimes f \otimes v_1 \wedge \ldots \wedge v_{i+1}) -\sum_{j=1}^{i+1} (-1)^{j+1}\beta(y)(h \otimes \widetilde{v}_j.f \otimes v_1 \wedge \ldots \wedge v_{i+1}) \\
=& -\sum_{j=1}^{i+1} (-1)^{j+1}\beta(\widetilde{v}_j.y)(h \otimes \widetilde{v}_j.f \otimes v_1 \wedge \ldots \wedge v_{i+1}-\sum_{j=1}^{i+1} (-1)^{j+1}\beta(y)(h \otimes \widetilde{v}_j.f \otimes v_1 \wedge \ldots \wedge v_{i+1})  
\end{align*}

We also define ${}^{\iota}d_{i,W}^{\vee, *}$ as follows: for $\beta \in \mathrm{Hom}_{W}( \mathrm{Res}_WY , (\mathrm{Res}_W\iota(X)^* \otimes \wedge^iV)^{\vee})$,
\begin{align*}
 & ({}^{\iota}d_{i,W}^{\vee, *}(\beta)(y))(f \otimes v_1 \wedge \ldots \wedge v_{i+1}) \\
=& -\sum_{j=1}^{i+1} (-1)^{j+1} \beta(\widetilde{v}_j.y)(f \otimes v_1 \wedge \ldots \wedge \widehat{v}_j \wedge \ldots \wedge v_{i+1})  - \sum_{j=1}^{i+1} (-1)^{j+1} \beta(y)(\widetilde{v}_j.f \otimes v_1 \wedge \ldots \wedge \widehat{v}_j \wedge \ldots \wedge v_{i+1})
\end{align*}

By direct computations, we have the following:
\begin{lemma} \label{lem commute d} Let $\eta \in \mathrm{Hom}_{\mathbb{H}}({}^{\iota}P_i, Y^*)$. Then
\begin{enumerate}
\item ${}^{\iota}d_i^{\vee,*}\eta^*=({}^{\iota}d_i^{\vee}\eta)^*$,
\item ${}^{\iota}d_{i,W}^{\vee, *} (T(\eta^*))=T({}^{\iota}d_i^{\vee,*}(\eta^*))$.
\end{enumerate}

\end{lemma}

\begin{lemma} \label{lem for B}
Let $x \in X$. Then
\begin{align*}
 & B_{i}(\sum_{j=1}^{i+1} (-1)^{j+1} \widetilde{e}_{k_j}.x \otimes e_{k_1} \wedge \ldots\wedge \widehat{e}_{k_j}\wedge \ldots \wedge e_{k_{i+1}})(f \otimes e_{l_1} \wedge \ldots \wedge e_{l_{n-i}}) \\
=& -(-1)^iB_{i+1}(x \otimes e_{k_1} \wedge \ldots \wedge e_{k_{i+1}})(\sum_{j=1}^{n-i} (-1)^{j+1} \widetilde{e}_{l_j}.f \otimes e_{l_1} \wedge \ldots \wedge \widehat{e}_{l_j} \wedge \ldots \wedge e_{l_{n-i}})
\end{align*}
\end{lemma}

\begin{proof}
Direct computation (similar to the proof of Lemma \ref{lem tr transfer D} below). We omit the details.
\end{proof}

\begin{proposition} \label{prop duality tr}
For any $\eta_1 \in \ker d_i^{\vee} \setminus \im d_{i-1}^{\vee}$, there exists $\eta_2 \in \ker {}^{\iota}d_{n-i}^{\vee}\setminus \im {}^{\iota}d_{n-i-1}^{\vee}$ such that $\mathrm{Tr}_i(\eta_2^{*} \circ \eta_1) \neq 0$. 
\end{proposition}

\begin{proof}
The proof is similar to the one of Theorem \ref{thm poin dua}. We briefly explain how to obtain the statement without going through all the computations. For simplicity set $Q_i =\mathrm{Res}_WX \otimes \wedge^iV$. We identify $\mathrm{Hom}_{W}(Q_i, Q_i)$ with $Q_i^{\vee} \otimes Q_i$. Then for $\sum_{k}  q_k^{\vee}\otimes q_k' \in Q_i^{\vee} \otimes Q_i$, 
\begin{align} \label{eqn trace formula}
\mathrm{trace}(\sum_{k} q_k^{\vee} \otimes q_k') =\sum_k q_k^{\vee}(q_k').
\end{align}

Let $\eta \in \mathrm{Hom}_{\mathbb{H}}(P_i, Y)$ and let $\eta' \in \mathrm{Hom}_{\mathbb{H}}({}^{\iota}P_{n-i-1}, Y^*)$. We regard $\eta \circ \vartheta_i$ as an element in $\mathrm{Hom}_W(\mathrm{Res}_WX \otimes \wedge^iV, \mathrm{Res}_WY)$ and ${}^{\iota}\vartheta_{n-i}^{*} \circ \eta'$ as an element in $\mathrm{Hom}_{W}( \mathrm{Res}_WY , (\mathrm{Res}_W\iota(X)^{\bullet} \otimes \wedge^iV)^{\vee})$. Using (\ref{eqn trace formula}) and Lemma \ref{lem for B}, we have $\mathrm{trace}(B_{i+1}^{-1}\circ ({}^{\iota}\vartheta_{n-i-1}^{*} \circ \eta') \circ (d_{i,W}^{\vee}(\eta \circ \vartheta_i))=(-1)^{i+1}\mathrm{trace}(B_{i}^{-1}\circ ({}^{\iota}d_{n-i-1,W}^{\vee, *}({}^{\iota}\vartheta_{n-i-1}^{*} \circ \eta') \circ (\eta \circ \vartheta_i))$. We also have $d_{i,W}^{\vee}(\eta \circ \vartheta_i)=(d_i^{\vee} \eta)\circ \vartheta_{i+1}$ and ${}^{\iota}d_{n-i-1,W}^{\vee, *}({}^{\iota}\vartheta_{n-i-1}^{*} \circ (\eta')^*)={}^{\iota}\vartheta_{n-i}^{*} \circ ({}^{\iota}d_{n-i-1}^{\vee}\eta')^*$ (e.g. by Lemma \ref{lem commute d} and (\ref{eqn dual Frobenius})). Combining the equations, we have
\begin{align} \label{eqn adjoint for trac}  \mathrm{Tr}_{i+1}((\eta')^* \circ (d^{\vee}_{i}\eta))=(-1)^{i}\mathrm{Tr}_i(({}^{\iota}d_{n-i-1}^{\vee}  \eta')^* \circ \eta)
\end{align}

Thus $\mathrm{Tr}_i$ replaces the non-degenerate bilinear pairing in Step 1 of the proof of Theorem \ref{thm poin dua} and the equation (\ref{eqn adjoint for trac}) replaces the computation of an adjoint operator in step 2 of the proof of Theorem \ref{thm poin dua}. Then with some linear algebra argument, we can obtain the statement.
\end{proof}

\subsection{Lemma for transferring information between $\mathrm{Tr}$} \label{ss lemma trace}
We retain the notation in Section \ref{ss bilinear pair}. The following lemma is an essential part for the proof of Theorem \ref{thm dua YP}.
\begin{lemma} \label{lem bilinear pair}
Let $\alpha: P_{i} \rightarrow {}^{\iota}P_{n-i-1}^*$ be an $\mathbb{H}$-map. Consider
\[\xymatrix{ 
 P_{i+1}\ar[r]^{d_i} &   P_{i} \ar[d]_{\alpha} & \\
                                             &  {}^{\iota}P_{n-i-1}^{*} \ar[r]^{{}^{\iota}d_{n-i-1}^*}  &  {}^{\iota}P_{n-i}^{*}  }.
\]
We have $\mathrm{Tr}_i({}^{\iota}d_{n-i-1}^* \circ \alpha) =-(-1)^i\mathrm{Tr}_{i+1}(\alpha \circ d_{i})$. 
\end{lemma}

\begin{proof}
Define $\mathcal D_i, \mathcal E_i: P_{i+1}\rightarrow P_i$ as follows:
\begin{align}
\mathcal D_{i}(h \otimes x \otimes v_1 \wedge \ldots \wedge v_{i+1}) &= \sum_{j=1}^{i+1} (-1)^{j+1} h\widetilde{v}_j \otimes x \otimes v_1 \wedge \ldots \wedge \widehat{v}_j \wedge \ldots \wedge v_{i+1} 
\end{align}
\begin{align}
\mathcal E_{i}(h \otimes x \otimes v_1 \wedge \ldots \wedge v_{i+1}) &= \sum_{j=1}^{i+1} (-1)^{j+1} h\otimes \widetilde{v}_j.x \otimes v_1 \wedge \ldots \wedge \widehat{v}_j \wedge \ldots \wedge v_{i+1} 
\end{align}
Both $\mathcal D_i$ and $\mathcal E_i$ are well-defined $W$-maps (by Lemma \ref{lem tilde element}). Moreover,
\[  d_i  = \mathcal D_i -\mathcal E_i.
\]
We similarly define the maps: ${}^{\iota}\mathcal D_i, {}^{\iota}\mathcal E_i: {}^{\iota}P_{i+1} \rightarrow  {}^{\iota}P_{i}$ and denote  ${}^{\iota}\mathcal D_i^*, {}^{\iota}\mathcal E_i^*: {}^{\iota}P_{i}^* \rightarrow  {}^{\iota}P_{i+1}^*$ for the daul maps. We also have ${}^{\iota}d_i^*={}^{\iota}\mathcal D_i^*-{}^{\iota}\mathcal E_i^*$.

We need some more notations. Let $\wp_i: P_i \rightarrow \mathrm{Res}_WX \otimes \wedge^iV$ be a linear map such that $\wp_i \circ \vartheta_i =\Id_Q$.

Recall that ${}^{\iota}\vartheta^{*}_i:{}^{\iota}P_i^{*}\rightarrow  (\mathrm{Res}_WX \otimes \wedge^iV)^{\vee} $. Let ${}^{\iota}\wp_i^{*}:  (\mathrm{Res}_WX \otimes \wedge^iV)^{\vee}\rightarrow {}^{\iota}P_i^*$ be a linear map such that ${}^{\iota}\vartheta^{*}_i \circ {}^{\iota}\wp_i^{*} =\Id$. 

It is straightforward to verify that 
\begin{align} \label{eqn for operator E 1} {}^{\iota}\vartheta^{*}_{n-i-1}\circ \alpha \circ \mathcal E_{i}\circ \vartheta_{i+1}={}^{\iota}\vartheta^{*}_{n-i-1}\circ  \alpha \circ \vartheta_{i}\circ \wp_{i} \circ\mathcal E_{i}\circ \vartheta_{i+1},
\end{align}
which follows from the fact that $\mathcal E_{n-i}$ sends the subspace $1 \otimes \mathrm{Res}_WX \otimes \wedge^{n-i+1}V$ to $1 \otimes \mathrm{Res}_WX \otimes \wedge^{n-i}V$.  Similarly, we have
\begin{align} \label{eqn for operator E 2} {}^{\iota}\vartheta^{*}_{n-i} \circ {}^{\iota}{\mathcal E}_{n-i-1}^{*} \circ \alpha \circ \vartheta_{i}={}^{\iota}\vartheta^{*}_{n-i} \circ {}^{\iota}{\mathcal E}_{n-i-1}^{*}\circ {}^{\iota}\wp_{n-i-1}^{*} \circ {}^{\iota}\vartheta^{*}_{n-i-1}\circ \alpha \circ \vartheta_{i} .\end{align} 
Now we have 
\begin{align*}
 & \mathrm{Tr}_{i+1}( \alpha \circ \mathcal E_{i})  \\
=& \mathrm{Tr}_{i+1}( \alpha \circ \vartheta_{i}\circ \wp_{i} \circ\mathcal E_{i}) \quad \mbox{ by (\ref{eqn for operator E 1}) }\\
=& -(-1)^{i}\mathrm{trace} ( (B_{i+1}^{-1}\circ {}^{\iota}\vartheta^{*}_{n-i-1}\circ \alpha \circ \vartheta_{n-i})\circ (B_{i}^{-1}  \circ {}^{\iota}\vartheta^{*}_{n-i} \circ {}^{\iota}{\mathcal E}_{n-i-1}^{*}\circ {}^{\iota}\wp_{n-i-1}^{*} \circ  B_{i+1})) \\
 &\quad \quad \quad \quad \quad \quad \quad \quad \quad \quad \quad \quad \quad \quad \quad \quad \quad \quad \mbox{(by Lemma \ref{lem tr tansfer E} below)} \\
=&  -(-1)^{i}\mathrm{trace} ((B_{i}^{-1}  \circ {}^{\iota}\vartheta^{*}_{n-i} \circ {}^{\iota}{\mathcal E}_{n-i-1}^{*}\circ {}^{\iota}\wp_{n-i-1}^{*} \circ  B_{i+1})\circ (B_{i+1}^{-1}\circ {}^{\iota}\vartheta^{*}_{n-i+1} \circ\alpha \circ \vartheta_{n-i}) ) \\
  &  \quad \quad \quad \quad \quad \quad \quad \quad \quad \quad \quad \quad \quad \quad \quad \quad \quad \quad  \mbox{(by the fact $\mathrm{trace}(MN)=\mathrm{trace}(NM)$)}\\
	=&  -(-1)^{i}\mathrm{trace} (B_{i}^{-1}  \circ {}^{\iota}\vartheta^{*}_{n-i} \circ {}^{\iota}{\mathcal E}_{n-i-1}^{*}\circ {}^{\iota}\wp_{n-i-1}^{*} \circ {}^{\iota}\vartheta^{*}_{i-1} \circ \alpha \circ \vartheta_{i})  \\
		=&  -(-1)^{i}\mathrm{trace} (B_{i}^{-1}  \circ {}^{\iota}\vartheta^{*}_{n-i} \circ {}^{\iota}{\mathcal E}_{n-i-1}^{*} \circ \alpha \circ \vartheta_{i}) \quad \mbox{ by (\ref{eqn for operator E 2})} \\
				=&  -(-1)^{i}\mathrm{Tr}_{i}( {}^{\iota}{\mathcal E}_{n-i-1}^{*} \circ \alpha)  
\end{align*}
Now combining with Lemma \ref{lem tr transfer D} below, we have $\mathrm{Tr}_i({}^{\iota}d_{n-i-1}^* \circ \alpha) =-(-1)^i\mathrm{Tr}_{i+1}(\alpha \circ d_{i})$. 
\end{proof}

\begin{lemma} \label{lem tr transfer D}
With the notations in the proof of Lemma \ref{lem bilinear pair}, 
\[ \mathrm{Tr}_{i}( {}^{\iota}\mathcal D_{n-i-1}^* \circ \alpha)=-(-1)^{i}\mathrm{Tr}_{i+1}(\alpha \circ \mathcal D_{i}) .\]
\end{lemma}

\begin{proof}
We fix a basis $x_1, \ldots, x_{\dim X}$ for $X$. Then $\left\{ x_r \otimes e_{l_1} \wedge \ldots \wedge e_{l_{i}} \right\}_{1 \leq r \leq \dim X, 1 \leq l_1 <\ldots <l_{i} \leq n}$ forms a basis for $\mathrm{Res}_WX \otimes \wedge^{i}V$ and $\left\{ x_r \otimes e_{l_1'} \wedge \ldots \wedge e_{l_{i+1}'} \right\}_{1 \leq r \leq \dim X, 1 \leq l_1' <\ldots <l_{i+1}' \leq n}$ forms a basis for $\mathrm{Res}_WX \otimes \wedge^{i+1}V$. Let $f_{x_1},\ldots, f_{x_{\dim V}}$ be a basis for $\mathrm{Res}_W\iota(X)^{\bullet}$ dual to $x_1, \ldots, x_{\dim X}$ in the sense that $f_{x_r}(x_s)=\delta_{r,s}$.

Let $\mathcal L$ (resp. $\mathcal L'$) be the collection of elements $(l_1, \ldots, l_{i})$ (resp. $(l_1',\ldots, l_{n-i+1}')$) in $\mathbb{Z}^{i}$ (resp. $\mathbb{Z}^{n-i+1}$) such that $1 \leq l_1 < \ldots< l_{i} \leq n$ (resp. $1 \leq l_1'<\ldots <l_{i+1}'\leq n$).  

For $L=(l_1, \ldots, l_{i}) \in \mathcal L$, fix integers $k_1^L,\ldots, k_{n-k}^L$ such that  
\[\langle e_{l_1} \wedge \ldots \wedge e_{l_{i}}, e_{k_1^L} \wedge \ldots \wedge e_{k_{n-i}^L} \rangle_{\wedge^{i}V}=1 .\]
For each $k_j^L$, let $N(k_j^L, L)$ be the unique integer such that $l_{N(k_j^L,L)-1}< k_j^L < l_{N(k_j^L, L)}$. For $L' \in \mathcal L'$, we similarly fix integers $k_1^{L'}, \ldots, k_{i-1}^{L'}$ and define $N(k_j^{L'}, L')$. 


We have
\begin{align*}
 & \mathrm{Tr}_{i}({}^{\iota}\mathcal D_{n-i-1}^* \circ \alpha) \\
=& \sum_{\substack{1 \leq r \leq \dim X \\ L=(l_1,\ldots,l_{i})\in \mathcal L}} {}^{\iota}\mathcal D_{n-i-1}^* \circ \alpha(1\otimes x_r \otimes e_{l_1} \wedge \ldots \wedge e_{l_{i}})(1 \otimes f_{x_r} \otimes e_{k_1^{L}} \wedge \ldots \wedge e_{k_{n-i}^{L}}) \\
=& \sum_{\substack{1 \leq r \leq \dim X \\ L=(l_1,\ldots,l_{i})\in \mathcal L}} \alpha(1 \otimes x_r \otimes e_{l_1} \wedge \ldots \wedge e_{l_{i}})( \sum_{j=1}^{n-i}(-1)^{j+1} \widetilde{e}_{k_j} \otimes f_{x_r} \otimes e_{k_1^{L}} \wedge \ldots \wedge \widehat{e}_{k_j^{L}}\wedge \ldots \wedge e_{k_{n-i}^{L}}) \\
=& -\sum_{\substack{1 \leq r \leq \dim X \\ L=(l_1,\ldots,l_{i})\in \mathcal L}} \sum_{j=1}^{n-i}(-1)^{j+1}\alpha(\widetilde{e}_{k_j} \otimes x_k \otimes e_{l_1} \wedge \ldots \wedge e_{l_{i}})( 1 \otimes f_{x_r} \otimes e_{k_1^{L}} \wedge \ldots \wedge \widehat{e}_{k_j^{L}}\wedge \ldots \wedge e_{k_{n-i}^{L}}) \\
 =& -(-1)^{i}\sum_{j=1}^{n-i}\sum_{\substack{1 \leq r \leq \dim X \\ L=(l_1,\ldots,l_{i})\in \mathcal L}}(-1)^{N(k_j, L)+1}  \alpha(\widetilde{e}_{k_j} \otimes x_r \otimes e_{l_1} \wedge \ldots  \wedge e_{l_{i}}) \\
  &    \quad \quad  \quad  \quad ( 1 \otimes f_{x_r} \otimes e^{l_1, \ldots, l_{N(k_j^L,L)-1},  l_{N(k_j^L,L)},\ldots,  l_{i}}_{k_1} \wedge  \ldots \wedge e^{l_1, \ldots, l_{N(k_j^L,L)-1},  l_{N(k_j^L,L)}, \ldots,  l_{i}}_{k_{n-i}}) \\
	 =& -(-1)^{i}\sum_{\substack{1 \leq r \leq \dim X \\ L'=(l_1',\ldots,l'_{i+1})\in \mathcal L'}}(-1)^{j+1}   \alpha(\widetilde{e}_{k_j} \otimes x_r \otimes e_{l_1'} \wedge \ldots \wedge \widehat{e}_{l_{j}'}\wedge \ldots \wedge e_{l_{i+1}'}) ( 1 \otimes f_{x_r} \otimes e_{k_1^{L'}} \wedge  \ldots \wedge e_{k_{n-i-1}^{L'}}) \\
	=& -(-1)^{i}\mathrm{Tr}_{i+1}(\alpha \circ \mathcal D_{i})
\end{align*}

The first equality follows from the fact that 
\begin{align} \label{eqn change basis tr}
 B(x_r \otimes e_{l_1} \wedge \ldots \wedge e_{l_{n-i}})(f_{x_{r'}} \otimes e_{k^L_1}\wedge \ldots \wedge e_{k^L_i} )&=\left\{ \begin{array}{cc} 1 & \mbox{ if $r=r'$ and $\left\{ l_1, \ldots, l_{n-i} \right\}=L$} \\ 0 & \mbox{ otherwise} \end{array} \right. ,
\end{align}
The last equality uses a similar equation as (\ref{eqn change basis tr}). The third equality uses the fact that $\widetilde{e}_l^{*}=-\widetilde{e}_l$ and $\alpha$ is an $\mathbb{H}$-map.
\end{proof}

\begin{lemma} \label{lem tr tansfer E}
With the notations in the proof of Lemma \ref{lem bilinear pair}, 
\[{}^{\iota}\vartheta^{*}_{n-i} \circ {}^{\iota}{\mathcal E}_{n-i-1}^{*}\circ {}^{\iota}\wp_{n-i-1}^{*} \circ  B_{i+1}     = -(-1)^{i} B_{i} \circ  \wp_{i} \circ\mathcal E_{i}\circ \vartheta_{i+1} .\]

\end{lemma}

\begin{proof}
\begin{align*}
 &({}^{\iota}\vartheta^{*}_{n-i} \circ {}^{\iota}{\mathcal E}_{n-i-1}^{*}\circ {}^{\iota}\wp_{n-i-1}^{*} \circ  B_{i+1}(x \otimes e_{k_1} \wedge \ldots \wedge e_{k_{i+1}}))(f \otimes e_{l_{1}} \wedge \ldots \wedge e_{l_{n-i}}) \\
=&  B_{i+1}(x \otimes e_{k_1} \wedge \ldots \wedge e_{k_{i+1}})(\sum_{j=1}^{n-i}(-1)^{j+1} \widetilde{e}_{l_j}.f \otimes e_{l_1} \wedge \ldots \wedge \widehat{e}_{l_j} \wedge \ldots \wedge e_{l_{n-i}}) \\
=&-\sum_{j=i}^n(-1)^{j+1}B_{i+1}( \widetilde{e}_{l_j}.x \otimes e_{k_1} \wedge \ldots \wedge e_{k_{i+1}}))(f \otimes e_{l_1} \wedge \ldots \wedge \widehat{e}_{l_j} \wedge \ldots \wedge e_{l_{n-i}}) \\
=&-(-1)^i B_{i}( \sum_{j=1}^{i+1}(-1)^{j+1}\widetilde{e}_{k_j}.x \otimes e_{k_1} \wedge \ldots \wedge \widehat{e}_{k_j} \wedge \ldots \wedge e_{k_{i+1}}))(f \otimes e_{l_1} \wedge \ldots \wedge \widehat{e}_{l_j} \wedge \ldots \wedge e_{l_{n-i}}) \\
=& -(-1)^{i} (B_{i} \circ  \wp_{i} \circ\mathcal E_{i}\circ \vartheta_{i+1}(x \otimes e_{k_1} \wedge \ldots \wedge e_{k_{i+1}}))(f \otimes e_{l_{1}} \wedge \ldots \wedge e_{l_{n-i}})
\end{align*}
In the second equality, we used the fact that $(\widetilde{v}_j.f)(x)=f(\iota(\widetilde{v}_j)^{\bullet}.x)=-f(\widetilde{v}_j.x)$ (where $\widetilde{v}_j$ acts on $x$ by the action of $\mathbb{H}$ on $X$). 
\end{proof}

\subsection{Duality for the Yoneda product}

 We retain the notations in Sections \ref{ss bilinear pair} and \ref{ss lemma trace}. Recall that $\mathbb{D}$ is defined in  Definition \ref{def dualizing module}.

\begin{theorem} \label{thm dua YP}
Let $\mathbb{H}$ be the graded affine Hecke algebra associated to a root datum $( R, V, R^{\vee}, V^{\vee}, \Pi)$ and a parameter function $k: \Pi \rightarrow \mathbb{C}$ (Definition \ref{def graded affine}). Let $X$ and $Y$ be finite-dimensional $\mathbb{H}$-modules. Let $n=\dim V$. Then 
\begin{enumerate}
\item $\mathrm{Ext}^n_{\mathbb{H}}(X, \mathbb{D}(X)) \neq 0$;
\item the Yoneda product
\[  \mathscr Y_i: \mathrm{Ext}^{n-i}_{\mathbb{H}}(Y, \mathbb{D}(X))\otimes \mathrm{Ext}^i_{\mathbb{H}}(X, Y) \rightarrow \mathrm{Ext}_{\mathbb{H}}^n (X, \mathbb{D}(X)) \]
is a nondegenerate pairing;
\item there exists a linear functional $\mathscr D_X: \mathrm{Ext}_{\mathbb{H}}^n(X, \mathbb{D}(X)) \rightarrow \mathbb{C}$ such that $\mathscr{D}_X \circ \mathscr Y_i$ agrees with the pairing in Theorem \ref{thm poin dua} via the identification $\mathrm{Ext}^{n-i}_{\mathbb{H}}(Y, \mathbb{D}(X)) \cong \mathrm{Ext}_{\mathbb{H}}^{n-i}(X^*, \iota(Y)^{\bullet})$ using Proposition \ref{prop equiv ext group}.
\end{enumerate}

\end{theorem}

\begin{proof}
For (1), using Theorem \ref{thm poin dua} and Proposition \ref{prop equiv ext group}, we have 
\[ \mathrm{dim} \mathrm{Ext}^n_{\mathbb{H}}(X, \mathbb{D}(X)) = \mathrm{dim} \mathrm{Ext}^n_{\mathbb{H}}(X^*, \iota(X^{\bullet}))= \dim \mathrm{Hom}(X, X) .
\]
Hence $\mathrm{Ext}^n_{\mathbb{H}}(X, \mathbb{D}(X)) \neq 0$. This shows (1).

We now consider (2). Let $\eta \in \mathrm{ker}d_i^{\vee} \setminus \mathrm{im}d_{i-1}^{\vee}$, which is a representative of a non-zero element in $\mathrm{Ext}_{\mathbb{H}}^i(X,Y)$. Let $P_{\bullet} \rightarrow X$ be a projective resolution of $X$. Then the associated long exact sequence has a commutative diagram with maps $F_k$ of the following form:
\[\xymatrix{ 
  & & P_i \ar[d]^{\eta} \ar[r] & P_{i-1} \ar[d]^{F_{i-1}}\ar[r]& \cdots \cdots \ar[r]  & P_0 \ar[d]^{F_0} \ar[r]& X \ar[d]^{id} \ar[r] &  0 \\
 \mathrm{Yon}(\eta): & 0   \ar[r] & Y \ar[r] & Z_i   \ar[r] & \cdots \cdots  \ar[r]  & Z_1  \ar[r]& X \ar[r] & 0 }
\]
To obtain the theorem, it is equivalent to show that there exists $\eta' \in \ker {}^{\iota}d_{n-i}^{\vee} \setminus \im {}^{\iota}d_{n-i-1}^{\vee}$ such that the Yoneda product of $\eta$ and $\eta'$ is non-zero. For the remaining proof for (2), we shall divide into three steps. \\

\noindent
{\bf Step 1: Construct an element in $\mathrm{Ext}_{\mathbb{H}}^{n-i}(Y, \mathbb{D}(X))$ and the Yoneda product with $\eta$} \\
By Proposition \ref{prop duality tr}, there exists $\eta' \in \ker {}^{\iota}d_{n-i}^{\vee} \setminus \im {}^{\iota}d_{n-i-1}^{\vee}$ such that $\mathrm{Tr}_i((\eta')^* \circ \eta) \neq 0$. Let ${}^{\iota}P_{\bullet} \rightarrow \iota(X)^{\bullet}$ be a projective resolution of $\iota(X)^{\bullet}$. The associated long exact sequence $\mathrm{Yon}(\eta')$ of $\eta'$ is of the following form:
\[\xymatrix{ 
  &0 & \ar[l]\iota(X)^{\bullet}\ar[d]^{id} \ar[l] & {}^{\iota}P_{0} \ar[d]\ar[l]& \cdots \cdots \ar[l]_{{}^{\iota}d_{0}}  & {}^{\iota}P_{n-i-1} \ar[d] \ar[l] & {}^{\iota}P_{n-i}  \ar[d]^{\eta'} \ar[l]_{{}^{\iota}d_{n-i-1}} &  \ldots \ar[l]\\
 \mathrm{Yon}(\eta'): & 0    & \iota(X)^{\bullet} \ar[l] & {}^{\iota}Z_{1}   \ar[l] & \cdots \cdots  \ar[l]_{f_1}  & {}^{\iota}Z_{n-i}  \ar[l]_{f_{n-i-1}}& Y^* \ar[l]_{f_{n-i}} & 0 \ar[l] }
\]
with those ${}^{\iota}Z_k$ and $f_k$ ($k=1, \ldots, n-i$) is defined from the pushout so that we have the following commutative diagram:
\[
\xymatrix{
& {}^{\iota}Z_k  & \mathrm{coker} f_{k+1}  \ar[l]_{p_k}  \\
&   {}^{\iota}P_{k-1}  \ar[u]_{q_k}   & {}^{\iota}P_{k} \ar[l]^{{}^{\iota}d_{k-1}}\ar[u]_{\eta_k' }},
\] 
where $f_{n-i+1}: 0 \rightarrow Y^*$ is defined as the trivial map, and $f_{k}: {}^{\iota}Z_{k+1} \rightarrow {}^{\iota}Z_{k}$ is defined  inductively as the composition of the projection ${}^{\iota}Z_{k+1} \rightarrow \mathrm{coker} f_{k+1}$ and $\mathrm{coker} f_{k+1} \stackrel{p_k}{\rightarrow} {}^{\iota}Z_k$, and $\eta_k'$ is defined as
\[ \eta_{k-1}': {}^{\iota}P_{k-1} \stackrel{q_k}{\rightarrow} {}^{\iota}Z_k \rightarrow \mathrm{coker} f_{k} \]
Here we set ${}^{\iota}Z_{n-i+1}=Y^*$. \\



Since $*$ is a contravariant exact functor, we then have the following exact sequence:
\begin{align} \label{eqn dual sequnece}
\xymatrix{ 
 & 0   \ar[r]  & \mathbb{D}(X)  \ar[r] & {}^{\iota}Z_{1}^*   \ar[r]^{f_{1}^*} & \cdots \cdots  \ar[r]  & {}^{\iota}Z_{n-i}^*  \ar[r]^{f_{n-i}^*}& Y \ar[r] & 0  }
\end{align}
Furthermore, ${}^{\iota}Z^*_{k}$ ($k=0,\ldots, n-i$) is isomorphic to the pull-back of ${}^{\iota}d_{k-1}^*: {}^{\iota}P_{k-1}^* \rightarrow {}^{\iota}P_{k}^*$ and $(\eta')^*: Y^* \rightarrow {}^{\iota}P_{k}^*$. In particular, we have commutative diagram of the following form:
\[
\xymatrix{
& {}^{\iota}Z_k^* \ar[d]^{q_k^*} \ar[r]^{p_k^*} &  \ker f_{k+1}^* \ar[d]^{(\eta'_k)^*} \\
&   {}^{\iota}P_{k-1}^*  \ar[r]^{{}^{\iota}d_{k-1}^*}  & {}^{\iota}P_{k}^*},
\] 
where $(\eta_k')^*$, $p_k^*$, $q_k^*$ and $f_k^*$ are the corresponding dual maps of $\eta_k'$, $p_k$, $q_k$ and $f_k$ respectively. Here we also identify naturally $\ker f_{k+1}^*$ with $(\mathrm{coker} f_{k+1})^*$.

From above, the Yoneda product of the associated long exact sequences of $(\eta')^*$ and $\eta$ is equivalent to the exact sequence of the  following  form:
\[\xymatrix{ 
S: & 0   \ar[r]  & \mathbb{D}(X)  \ar[r]  & {}^{\iota}Z_{1}^*   \ar[r]  & \cdots \cdots  \ar[r]  & {}^{\iota}Z_{n-i}^*  \ar[r]& Z_i \ar[r]  &  \cdots \cdots  \ar[r]  & Z_1\ar[r]& X \ar[r] & 0 
}
\]

In order to show the Yoneda product of $(\eta')^*$ and $\eta$ gives rise a non-zero element in $\mathrm{Ext}^n_{\mathbb{H}}(X, \mathbb{D}(X))$, we have to construct a morphism from the sequence $P_{\bullet} \rightarrow X \rightarrow 0$ to the long exact sequence $S$: 
\[\xymatrix{ 
 & 0 \ar[r]& P_n \ar@{.>}[d] \ar[r] & P_{n-1} \ar@{.>}[d] \ar[r] & \cdots \cdots\ar[r] & P_i \ar@{.>}[d] \ar[r] & P_{i-1} \ar[d]^{F_{i-1}}\ar[r]& \cdots \cdots \ar[r]  & P_0 \ar[d]^{F_0} \ar[r]& X \ar[d]^{id} \ar[r] &  0 \\
S: & 0   \ar[r]  & \mathbb{D}(X)  \ar[r]  & {}^{\iota}Z_{1}^*   \ar[r]  & \cdots \cdots  \ar[r]  & {}^{\iota}Z_{n-i}^*  \ar[r]& Z_i \ar[r]  &  \cdots \cdots  \ar[r]  & Z_1\ar[r]& X \ar[r] & 0 
}
\]
Set ${}^{\iota}Z_0^*=\ker f_1^* \cong \mathbb{D}(X)$ for convenience. We next construct maps $F_k: P_k \rightarrow {}^{\iota}Z^*_{n-k}$ ($k=i, \ldots, n$) so that the above diagram is commutative and carries information about $\mathrm{Tr}$ for those $F_k$. Using a property of $\mathrm{Tr}_n$ for $F_n$ , we shall conclude $S=\mathrm{Yon}_{X, \mathbb{D}(X)}^n(F_n)$ is a non-zero element.  \\

\noindent
{\bf Step 2: Construct maps $F_k: P_k \rightarrow {}^{\iota}Z_{n-k}^*$ } 

Consider the following commutative diagram:
\[\xymatrix{ 
  &   P_{i+1} \ar[r]^{d_i}  &  P_i \ar@{.>}[d]^{\exists \widetilde{F}_i} \ar[r]^{\eta} &   Y\ar[d]^{(\eta')^*}  \\
 &                    &   \ar[r]^{{}^{\iota}d_{n-i-1}^*}     {}^{\iota}P_{n-i-1}^*           &  {}^{\iota}P_{n-i}^*   }
\]
We have 
\begin{enumerate}
\item $\im (\eta')^* \subset \im {}^{\iota}d_{n-i-1}^*$ because $\ker{}^{\iota}d_{n-i}^* =\im{}^{\iota}d_{n-i-1}^*$ and $(\eta')^* \in \ker {}^{\iota}d_{n-i}^{\vee, *}$.
\item By (1) and $P_i$ being projective, there exists a map $\widetilde{F}_i: P_i \rightarrow  {}^{\iota}P_{n-i-1}^* $ such that ${}^{\iota}d_{n-i-1}^* \circ \widetilde{F}_i=(\eta')^* \circ \eta$. 
\item By the universal property of pullback, there exists $F_i: P_i \rightarrow  {}^{\iota}Z_{n-i}^*$ with the following commutative diagram:
\[
\xymatrix{
P_i \ar@/_/[ddr]_{\widetilde{F}_i} \ar@/^/[drr]^{\eta} \ar@{.>}[dr]|-{\exists F_i}\\
& {}^{\iota}Z_{n-i}^* \ar[d]^{q_{n-i}^*} \ar[r]_{p_{n-i}^*} & Y\ar[d]^{(\eta')^*} \\
&   {}^{\iota}P_{n-i-1}^*  \ar[r]^{{}^{\iota}d_{n-i-1}^*}  & {}^{\iota}P_{n-i}^*}
\]
\item We have $ \im (F_i \circ d_i)   \subset \ker f_{n-i}^*$ because $(p_i^*\circ F_i) \circ d_i=\eta \circ d_i=0$, where the last equality follows from $\eta \in \ker d_{i}^{\vee}$. 

\item By (4) and $P_{i+1}$ being projective, there exists a map $\eta_{i+1}: P_{i+1} \rightarrow \ker f_{n-i}^*$ such that the following diagram is commutative:
\[\xymatrix{
P_{i+1} \ar@{.>}[r]^{\exists \eta_{i+1}} \ar[rdd]_{\widetilde{F}_i \circ d_i} \ar[rd]^{F_i \circ d_i} & \ker f_{n-i}^*  \ar[d]^{I_{n-i}} \\
&  {}^{\iota}Z_{n-i}^* \ar[d]^{q_{n-i}^*} \\
& {}^{\iota}P_{n-i-1}^*} ,
\]
where $I_{n-i}$ is the natural inclusion map. By definition, $q_{n-i}^* \circ I_{n-i}=(\eta_{n-i-1}')^*$.
\item $\mathrm{Tr}_{i+1}((\eta_{n-i-1}')^* \circ \eta_{i+1})=\mathrm{Tr}_{i+1}(\widetilde{F}_i \circ d_i)= \mathrm{Tr}_i( {}^{\iota}d_{n-i-1}^{*} \circ \widetilde{F}_i )=\mathrm{Tr}_i((\eta')^* \circ \eta) \neq 0  $, where the first equality follows from (5) and the second equality follows from Lemma \ref{lem bilinear pair}.
\item $\eta_{i+1} \in \ker d_{i+1}^{\vee}$ because $I_{n-i} \circ \eta_{i+1} \circ d_{i+1}=F_i \circ d_i \circ d_{i+1} =0$ and $ I_{n-i}$ is injective. Moreover, by definition, $(\eta_{n-i-1}')^* \in \ker {}^{\iota}d_{n-i-1}^{\vee, *}$ because ${}^{\iota}d_{n-i-1}^{ *}\circ (\eta_{n-i-1}')^* =(\eta')^*\circ p_{n-i}^*\circ I_{n-i} $ and $p_{n-i}^*\circ I_{n-i}=0$ (by definition of $f_{n-i}^*$ for the last equation). (By abuse of notation, here the map $d_{i+1}^{\vee}$ denotes for a map from $\mathrm{Hom}_{\mathbb{H}}(\mathbb{H} \otimes_{\mathbb{C}[W]}(\mathrm{Res}_WX \otimes \wedge^{i+1} V), \ker f_{n-i}^*   )$ to $\mathrm{Hom}_{\mathbb{H}}(\mathbb{H} \otimes_{\mathbb{C}[W]}(\mathrm{Res}_WX \otimes \wedge^{i} V), \ker f_{n-i}^*)$.)
\item Now $\ker f_{n-i}^*$ replaces the role of $Y$; $(\eta_{n-i-1}')^*$ and $\eta_{i+1}$ replace the roles of $(\eta')^*$ and $\eta$ respectively; $P_{i+1}$ and ${}^{\iota}P_{n-i-1}$ replaces the roles of $P_i$ and ${}^{\iota}P_{n-i}$ respectively. With (6) and (7), we can repeat the argument from (1) to (7). Inductively, we obtain maps, for $k=i, \ldots, n-1$, $\eta_{k+1}: P_{k+1} \rightarrow \mathrm{ker}f_{n-k}^*$ and $\widetilde{F}_k: P_k \rightarrow {}^{\iota}P_{n-k-1}^*$ and $F_k: P_k \rightarrow {}^{\iota}Z_{n-k}^*$. Furthermore, we have $\mathrm{Tr}_k((\eta_{n-k-1}')^* \circ \eta_{k+1}) \neq 0$. 
\item From the map $\eta_n: P_n \rightarrow \ker f_1^*$ and $\ker f_1^* \cong \mathbb{D}(X)$, we obtain a map $F_n: P_n \rightarrow \mathbb{D}(X)$. 
\item From our construction, $F_i$ completes the commutative diagram. \\

\end{enumerate}

\noindent
{\bf Step 3: Show that $S=\mathrm{Yon}_{X, \mathbb{D}(X)}^n(F_n) \neq 0$ }

It is equivalent to show $\mathrm{Yon}_{X, \ker d_0^*}^n(\eta_n)\neq 0$. Suppose 
\[\eta_n \in \mathrm{im} (d_{n-1}^{\vee}: \mathrm{Hom}_{\mathbb{H}}(\mathbb{H} \otimes_{\mathbb{C}[W]}(\mathrm{Res}_WX \otimes \wedge^{n-1} V), \ker f_1^*) \rightarrow \mathrm{Hom}_{\mathbb{H}}(\mathbb{H} \otimes_{\mathbb{C}[W]}(\mathrm{Res}_WX \otimes \wedge^{n} V), \ker f_1^*) )
\]
Then $\eta_n=d^{\vee}_{n-1} (\overline{\eta})$ for some $\overline{\eta} \in  \mathrm{Hom}_{\mathbb{H}}(\mathbb{H} \otimes_{\mathbb{C}[W]}(\mathrm{Res}_WX \otimes \wedge^{n-1} V), \ker d_0^*) $. Then $\mathrm{Tr}_n( (\eta_0')^*\circ d_{n-1}^{\vee}\overline{\eta})= \mathrm{Tr}_{n-1}(({}^{\iota}d_0^{\vee,*}( (\eta_0')^*)\circ \overline{\eta} )=0$ by equation (\ref{eqn adjoint for trac}) and $(\eta_0')^* \in  \ker f_1^{*}$, giving a contradiction to our construction (see (6) and (8) in step 2). This implies that $\eta_n \notin \im d_{n-1}^{\vee}$ and hence $\mathrm{Yon}_{X, \ker d_0^*}^n(\eta_n)\neq 0$ as desired. This completes the proof for (2).

We now consider (3).  Via the exact sequence in (\ref{eqn dual sequnece}), there is a natural isomorphism $f_0^*: \mathbb{D}(X) \rightarrow \mathrm{ker} f_1^*$. Define the map $\mathscr D_X: \mathrm{Ext}_{\mathbb{H}}^n(X, \mathbb{D}(X)) \rightarrow \mathbb{C}$ as follows: for $\eta_n \in \mathrm{Hom}_{\mathbb{H}}(\mathbb{H} \otimes_{\mathbb{C}[W]}(\mathrm{Res}_WX \otimes \wedge^n V), \mathbb{D}(X)) $, define
\[ \mathscr D_X( \eta_n)  = \mathrm{Tr}_n( {}^{\iota}d_{-1}^{\vee,*} \circ \eta_n ), 
\]
where ${}^{\iota}d_{-1}^{\vee, *}: \mathbb{D}(X) \rightarrow {}^{\iota}P_{0}^*$ is the natural map induced from the surjective map ${}^{\iota}P_0 \rightarrow \iota(X^{\bullet})$. 
Note that
\begin{align} \label{eqn alt form DX}
 \mathscr D_X( \eta_n)  = \mathrm{Tr}_n((\eta_0')^*\circ  f_0^*\circ \eta_n ) .
\end{align}
by using a commutative diagram. Then using the argument given in the first paragraph of Step 3 above, we can show that $\mathscr D_X( \eta)  = \mathrm{Tr}_n((\eta_0')^*\circ  f_0^*\circ \eta)=0$ for $\eta \in \mathrm{im}\ d_{n-1}^{\vee}$. Recall that
\[  \mathrm{Ext}^n_{\mathbb{H}}(X, \mathbb{D}(X)) \cong \frac{ \mathrm{Hom}_{\mathbb{H}}(\mathbb{H} \otimes_{\mathbb{C}[W]}(\mathrm{Res}_WX \otimes \wedge^n V), \mathbb{D}(X)) }{\mathrm{im}\ d_{n-1}^{\vee} }.
\]
Hence $\mathscr D_X$ descents to a map $\mathrm{Ext}_{\mathbb{H}}^n(X, \mathbb{D}(X)) \rightarrow \mathbb{C}$ (i.e. independent of choice of a representative). It remains to see $\mathscr D_X \circ \mathscr Y_i$ coincides with the pairing in Theorem \ref{thm poin dua}. To this end, using (\ref{eqn alt form DX}) and Step 2 (6) and (8) above, we have 
\[ \mathrm{Tr}_i((\eta')^* \circ \eta)= \mathrm{Tr}_n((\eta_0')^*\circ  f_0^*\circ \eta_n )=\mathscr D_X( \eta_n) ,
\]
where $\eta_n =\mathscr{Y}_i(\eta', \eta)$. (Note that in our construction for (2), we assume $\eta'$ is specially picked in Step 1. In fact, all the construction can still work for arbitrary $\eta'$ and the only thing we do not have is $\mathrm{Tr}_i((\eta')^* \circ \eta)\neq 0$.) From our construction in Section \ref{ss bilinear pair}, we have $\mathrm{Tr}_i$ is a non-zero scalar multiple of the bilinear form in Theorem \ref{thm poin dua}. This proves (3).
\end{proof}



\begin{example} \label{exam ext indecom}
\begin{enumerate}
\item[(1)] When $\mathbb{H}=S(V)$, $\mathrm{Ext}^*_{S(V)}(\mathbb{C}, \mathbb{C}) \cong \wedge^*V$, where $\mathbb{C}$ is regarded as the trivial $S(V)$-module i.e. $v. x=0$ for all $x \in \mathbb{C}$. The Yoneda product agrees with the natural product structure on $\wedge^*V$. 
\item[(2)] Let $R$ be of type $G_2$, let $V$ be the space spanned by $R$, and let $k \equiv 1$. Let $\alpha, \beta$ be the simple roots of $R$ with $\langle \alpha, \beta^{\vee} \rangle =-1$ and $\langle \beta, \alpha^{\vee} \rangle=-3$. We consider modules of the central character $\alpha^{\vee}+\beta^{\vee}$, in which the tempered modules correspond to the subregular nilpotent orbit under the Kazhdan-Lusztig classification. For such central character, there exists a unique irreducible discrete series, denoted $DS$, which contains the sign representation as a $\mathbb{C}[W]$-module, and there exists a unique simple module, denoted by $Z$ which is invariant under the Iwahori-Matsumoto involution $\iota$. We have 
\[ 
\mathrm{Ext}^i_{\mathbb{H}}(DS, Z) \cong \mathrm{Ext}^i_{\mathbb{H}}(Z, DS)\cong  \left\{ \begin{array}{cc} \mathbb{C} & \mbox{ if $i=1$ } \\ 0 & \mbox{ otherwise } \end{array}  \right.
\]
(see \cite[Section 8.2]{Ch}). Then we have non-trivial extensions of $DS$ by $Z$ and $Z$ by $DS$ respectively as follows:
\[ 0\rightarrow DS \rightarrow X_1 \rightarrow Z \rightarrow 0, \quad 0 \rightarrow Z \rightarrow X_2 \rightarrow DS \rightarrow 0 .\]
By considering the Yoneda product and Theorem \ref{thm dua YP}, we obtain a long exact sequence:
\[ 0 \rightarrow Z \rightarrow X_2 \rightarrow X_1 \rightarrow Z \rightarrow 0 ,\]
corresponding to a non-zero element in $\mathrm{Ext}^2_{\mathbb{H}}(Z, Z)$. (However, if we reverse the order of multiplication in the Yoneda product, then the long exact sequence
\[  0 \rightarrow DS \rightarrow X_1 \rightarrow X_2 \rightarrow DS \rightarrow 0 \]
corresponds to the zero element in $\mathrm{Ext}_{\mathbb{H}}^2(DS, DS)$. )

\end{enumerate}
\end{example}

Here is an application of Theorem \ref{thm dua YP}, which cannot be merely deduced from Theorem \ref{thm poin dua}.

\begin{corollary} 
In the notation of Example \ref{exam ext indecom}, there is no indecomposable module of length $3$ whose composition factors are two copies of $Z$ and one copy of $DS$.
\end{corollary}
\begin{proof}
 To show this, suppose there exists such module, say $M$ of length $3$ with those composition factors. Then one may consider all the possible radical filtration for $M$ to obtain a contradiction. Here we only consider that $M$ has radical filtration $\mathrm{rad}^i$: $\mathrm{rad}^0/\mathrm{rad}^1\cong Z$, $\mathrm{rad}^1/\mathrm{rad}^2\cong DS$, $\mathrm{rad}^2\cong Z$ and we shall show such $M$ is impossible to exist. (Other cases are relatively easier.)  We have the exact sequence:
\[  0 \rightarrow Z \rightarrow \mathrm{rad}^1 \rightarrow DS \rightarrow 0.
\]
Applying the $\mathrm{Hom}_{\mathbb{H}}(Z,.)$-functor, we have the following long exact sequence:
\[ \ldots \rightarrow \mathrm{Ext}^1_{\mathbb{H}}(Z, Z) \rightarrow \mathrm{Ext}^1_{\mathbb{H}}(Z,\mathrm{rad}^1) \rightarrow \mathrm{Ext}^1_{\mathbb{H}}(Z,DS)\stackrel{\partial}{\rightarrow} \mathrm{Ext}^2_{\mathbb{H}}(Z,Z ) \rightarrow \ldots \]
The connecting homomorphism $\partial$ is the Yoneda product $\mathrm{Ext}^1_{\mathbb{H}}(DS,Z) \otimes $ \cite[Ch III Theorem 9.1]{Ma} and hence according to the discussion in Example \ref{exam ext indecom} (which uses Theorem \ref{thm dua YP}), $\partial$ is injective. Thus $\mathrm{Ext}^1_{\mathbb{H}}(Z,\mathrm{rad}^1)=0$, which contradicts to the existence of such $M$.
\end{proof}

\section{Ind-Res resolution and Aubert-type involution} \label{s kato resolution}

Motivated from studies in \cite[Section 10]{Pr}, \cite[Section III.3, Section IV.4]{SS} and \cite{Ka0}, it is natural to relate the dual module $\mathbb{D}(X)$ (Definition \ref{def dualizing module}) with an involution arising from Ind-Res functors. Such involution is an analogue of an involution for $p$-adic groups studied in \cite{Au} by Aubert. The Ind-Res functors also give rise a resolution which will be used to compute extensions of discrete series in next section. 

The approach in this section is very close to the one for affine Hecke algebras in \cite{Ka0} by Kato.

\subsection{Coxeter complex} \label{ss cox compl}
In this section, we review a complex on a sphere which arises from the hyperplane arrangement of the reflection representation. 

Let $V'$ be the $W$-subspace of $V$ spanned by $\alpha \in \Pi$ so that $V'$ is the reflection representation of $W$. Let ${\mathcal FC}$ be the closed fundamental chamber of the action of $W$ on $V'$ i.e.
\[   {\mathcal FC}= \left\{ v \in V' : \langle v, \alpha^{\vee} \rangle \geq 0 \mbox{ for all $\alpha \in \Pi$ } \right\}. \]
Let $m=|\Pi|$. We project naturally the $\mathcal FC$ to the unit $(m-1)$-dimensional sphere in $V'$, and for $\alpha \in \Pi$, we denote $v_{\alpha} \in V'$ to be the image of the projection of the line $\bigcap_{\beta \in \Pi \setminus \left\{ \alpha \right\}} H_{\beta}$,  where $H_{\beta}$ is the hyperplane perpendicular to $\beta$. Write $\Pi=\left\{ \alpha_1, \ldots, \alpha_{m} \right\}$, and by which, we also have fixed an ordering of simple roots. Set $v_i=v_{\alpha_i}$. Then we also have an ordering on vertices $v_{1}, \ldots, v_{m}$ and hence obtain an ordered $(m-1)$-simplex $[v_{1}, \ldots, v_{m}]$ from the projection of $\mathcal{FC}$. By taking $W$-actions, we obtain other $(m-1)$-simplexes with ordered vertices $[w(v_1), \ldots, w(v_{m})]$ ($w \in W$) on the unit $(m-1)$-sphere. Those $(m-1)$-simplexes form an ordered simplicial complex on the $(m-1)$-sphere. 

We need some more notations for next subsection. For each $J \subset \Pi$, write $\Pi \setminus J=\left\{ \alpha_{i_1^J}, \ldots, \alpha_{i_k^J}\right\}$ with $i_1^J < \ldots < i_k^J$. Let $\Delta^J=[v_{i_1^J}, \ldots, v_{i_k^J}]$ be an $(m-|J|-1)$-simplex. For $J \subset J' \subset \Pi$ with $|J'|=|J|+1$, let $j$ be the unique index (depending on $J$ and $J'$) such that $\alpha_{i_j^{J'}} \in \left\{ \alpha_{i_1^J}, \ldots, \alpha_{i_{m-|J|}^J} \right\} \setminus \left\{ \alpha_{i_1^{J'}}, \ldots, \alpha_{i_{m-|J'|}^{J'}} \right\} $. We then set $\epsilon_J^{J'}= (-1)^{j+1}$.

Let $C_k^{\Delta}(X)$ be the free abelian group of the $k$-simplexes $w(\Delta^J):=[w(v_{i_1^J}), \ldots, w(v_{i_{k+1}^J})]$ for all $J \subset \Pi$ with $|J|=m-k-1$ and all $w\in W$.

\subsection{Notation for parabolic subalgebras of $\mathbb{H}$}
\begin{notation}
For any subset $J$ of $\Pi$, define $V_J$ to be the complex subspace of $V$ spanned by vectors in $J$ and define $V_J^{\vee}$ be the dual space of $V_J$ lying in $V^{\vee}$. Let $R_J = V_J \cap R$ and let $R_J^{\vee}=V_J^{\vee} \cap R^{\vee}$. Let $W_J$ be the subgroup of $W$ generated by the elements $s_{\alpha}$ for $\alpha \in J$. 
For $J \subset \Pi$, let $W_J$ be the subgroup of $W$ generated by all $s_{\alpha}$ with $\alpha \in J$. Let $w_{0, J}$ be the longest element in $W_J$. Let $W^J$ be the set of minimal representatives in the cosets in $W/W_J$.

For $J \subset \Pi$, let $\mathbb{H}_J$ be the subalgebra of $\mathbb{H}$ generated by all $v \in V$ and $t_w$ ($w \in W_J$). 
Note $\mathbb{H}_J$ is the graded affine Hecke algebra associated to the root datum $( R_J, V,  R_J^{\vee}, V^{\vee}, J)$.

\end{notation}

\subsection{Ind-Res resolution} \label{ss ind res resol} We keep using the notation in Section \ref{ss cox compl}.
Let $X$ be an $\mathbb{H}$-module. For $J \subset \Pi$, let 
\[  C_J(X) = \mathrm{Ind}_{\mathbb{H}_J}^{\mathbb{H}} \mathrm{Res}_{\mathbb{H}_J} X := \mathbb{H} \otimes_{\mathbb{H}_J} (\mathrm{Res}_{\mathbb{H}_J}X) .
\]
Let 
\[  C_i(X) = \bigoplus_{|J|=i} C_J(X) .
\]
 For $J \subset J' \subset \Pi$, define the natural map
\[\pi_J^{J'}: \mathrm{Ind}_{\mathbb{H}_J}^{\mathbb{H}} \mathrm{Res}_{\mathbb{H}_J} X  \rightarrow \mathrm{Ind}_{\mathbb{H}_{J'}}^{\mathbb{H}} \mathrm{Res}_{\mathbb{H}_{J'}} X , \quad h \otimes x \mapsto \epsilon_J^{J'} h \otimes x ,
\]
where $\epsilon_J^{J'}$ is an appropiate choice of $\pm 1$. 
Define
\[\pi_i= \bigoplus_{|J|=i} \bigoplus_{\substack{J \subset J' \subset \Pi \\ |J'|=i+1}} \pi_J^{J'}: C_i(X) \rightarrow C_{i+1}(X).
\]

Using $C_m(X)=X$, we have the following sequence of maps, which will be proven to be exact:
\begin{align} \label{eqn ind-res resol} 0\rightarrow \ker \pi_0 \rightarrow C_0(X) \stackrel{\pi_0}{\rightarrow} C_{1}(X) \stackrel{\pi_{1}}{\rightarrow} \ldots \stackrel{\pi_{m-2}}{\rightarrow} C_{m-1}(X) \stackrel{\pi_{m-1}}{\rightarrow} X \rightarrow 0 .
\end{align}

We now fix a non-zero $z \in X$. Let $C_J(X,z)$ be the linear subspace of $\mathbb{H} \otimes_{\mathbb{H}_J} \mathrm{Res}_{\mathbb{H}_J} X$ spanned by $t_w \otimes t_w^{-1}.z$ ($w \in W$). Let $C_i(X,z) \cong \bigoplus_{|J|=i} C_J(X,z)$ be the natural linear subspace of $C_i(X)$. Then $\pi_i$ induces a linear map $\pi_i^z: C_i(X,z) \rightarrow C_{i-1}(X,z)$. We now define another linear isomorphism $\varpi_J:   C_J(X,z)\rightarrow C_J^{\Delta}(X)$ characterized by
\[ t_w \otimes t_w^{-1}.z \mapsto [ w(v_{i_1^J}), \ldots, w(v_{i_k^J}) ] .
\]
Then those $\varpi_J$ induce a map $\varpi_i: C_i(X, z) \rightarrow C_i^{\Delta}(X)$. It is straightforward to verify that $\varpi_{i+1}^{-1} \circ \pi_i \circ \varpi_i$ is the differential for the simplicial complexes formed by $C_i^{\Delta}(X)$. Thus by the (reduced) homology of a sphere,
\begin{align*} \label{eqn ind-res resol x} 0\rightarrow \ker \pi_0^z \rightarrow C_0(X,z) \stackrel{\pi_0^z}{\rightarrow} C_{1}(X,z) \stackrel{\pi_{1}^z}{\rightarrow} \ldots \stackrel{\pi_{m-2}^z}{\rightarrow} C_{m-1}(X,z) \stackrel{\pi_{m-1}^z}{\rightarrow} C_m(X, z) \rightarrow 0 .
\end{align*}
is a long exact sequence. Hence, we have:

\begin{proposition} \label{prop ind res resol}
Let $X$ be an $\mathbb{H}$-module. (\ref{eqn ind-res resol}) is a resolution for $X$ (i.e.  (\ref{eqn ind-res resol}) is a long exact sequence).

\end{proposition}


\subsection{Aubert involution}
For any $x \in X$, define a linear map $\chi: X \rightarrow \mathbb{H} \otimes_{S(V)} \mathrm{Res}_{S(V)}X$ characterized by $\chi(x)=\sum_{w \in W} (-1)^{l(w)} t_w \otimes t_w^{-1}.x$. Note that $\chi$ is injective.

\begin{lemma} \label{lem im D action}
Let $x \in X$. 
\begin{enumerate} 
\item[(1)] For a simple reflection $s$, $t_s.\chi(x)=-\chi(t_s.x)=\chi(\iota(t_s^{\bullet})^*.x)$.
\item[(2)] For $v \in V$, $v. \chi(x)= \chi(\iota(v^{\bullet})^*.x)$. 
\end{enumerate}
In particular, $\im \chi$ is $\mathbb{H}$-invariant and $\im \chi \cong \mathbb{D}(X)$.
\end{lemma}

\begin{proof}
(1) is straightforward.
For (2), 
\begin{align*}
 & v. \chi(x) \\
= &  \sum_{w \in W}(-1)^{l(w)} vt_w \otimes t_w^{-1}.x \ \\
= & \sum_{w \in W} (-1)^{l(w)}vt_w \otimes t_w^{-1}.x \\
= & \sum_{w \in W} (-1)^{l(w)}t_w \left(w^{-1}(v)-\sum_{\alpha>0, w(\alpha)<0} k_{\alpha}\langle w^{-1}(v), \alpha^{\vee} \rangle t_{s_{\alpha}} \right) \otimes t_w^{-1}.x\\
= & \sum_{w \in W}(-1)^{l(w)} t_w \otimes w^{-1}(v) t_{w^{-1}}.x-\sum_{w \in W}\sum_{\alpha>0, w(\alpha)<0} (-1)^{l(w)}k_{\alpha}\langle w^{-1}(v), \alpha^{\vee} \rangle  t_wt_{s_{\alpha}} \otimes t_{w}^{-1}.x \\
= & \sum_{w \in W }(-1)^{l(w)} t_w \otimes w^{-1}(v) t_{w}^{-1}.x -\sum_{w \in W} \sum_{\alpha>0, w(\alpha)>0} (-1)^{l(ws_{\alpha})+1} k_{\alpha} \langle w^{-1}(v), \alpha^{\vee} \rangle t_w  \otimes t_{s_{\alpha}}t_{w}^{-1}.x
\end{align*}

On the other hand, we also have
\begin{align*}
 &  \chi (t_{w_0}\theta(v)t_{w_0}^{-1}.x )\\
= &  \sum_{w \in W}(-1)^{l(w)} t_w \otimes t_w^{-1}t_{w_0}\theta(v)t_{w_0}^{-1}.x\\
= & \sum_{w \in W} (-1)^{l(w)}t_w \otimes t_{w^{-1}w_0}\theta(v)t_{w_0}^{-1}.x \\
= & \sum_{w \in W} (-1)^{l(w)}t_w  \otimes\left(-w^{-1}(v)-\sum_{\alpha>0, w_0w(\alpha)<0} k_{\alpha}\langle -w^{-1}(v), \alpha^{\vee} \rangle t_{s_{\alpha}} \right) t_w^{-1}.x \quad \mbox{(by $w_0\theta(v)=-v$)} \\
= & -\sum_{w \in W}(-1)^{l(w)} t_w \otimes w^{-1}(v) t_{w}^{-1}.x-\sum_{w \in W} (-1)^{l(w)}\sum_{\alpha>0, w(\alpha)>0}k_{\alpha}\langle -w^{-1}(v), \alpha^{\vee} \rangle  t_w\otimes t_{s_{\alpha}} t_{w}^{-1}.x \\
= & -\sum_{w \in W }(-1)^{l(w)} t_w \otimes w^{-1}(v) t_{w}^{-1}.x +\sum_{w \in W} (-1)^{l(w)}\sum_{\alpha>0, w(\alpha)>0} k_{\alpha} \langle w^{-1}(v), \alpha^{\vee} \rangle t_w  \otimes t_{s_{\alpha}}t_w^{-1}.x \\
=&-v. \chi(x)
\end{align*}
This proves (2).
\end{proof}

\begin{lemma} \label{lem des ker}
With the notations in Section \ref{ss ind res resol}, $\ker \pi_0 =\im \chi$.
\end{lemma}

\begin{proof}
This is \cite[Lemma 1]{Ka0}. Note that if the element of the form $ \sum_{w \in W} a_wt_w \otimes (t_w^{-1}.x)$ is in $\ker \pi_0$, then $a_{ws}=-a_w$ for all $w \in W$ and all simple reflections $s \in W$. Hence $ \sum_{w \in W} a_wt_w \otimes (t_w^{-1}.x)=a_1\chi(x)$. We refer the reader for the detail in \cite[Lemma 1]{Ka0}. 
\end{proof}

Let $\mathrm{G}(\mathbb{H})$ be the Grothendieck group of finite-dimensional $\mathbb{H}$-modules. For a finite-dimensional $\mathbb{H}$-module $X$, denote by $[X]$ the image of $X$ in $\mathrm{G}(\mathbb{H})$. Define the Aubert-type involution $\widetilde{\mathbb{D}}: \mathrm{G}(\mathbb{H}) \rightarrow \mathrm{G}(\mathbb{H})$ as:
\[   \widetilde{\mathbb D}([X]) =\sum_{ J \subset \Pi} (-1)^{|J|} [\mathrm{Ind}_{\mathbb{H}_J}^{\mathbb{H}} \mathrm{Res}_{\mathbb{H}_J}X ].
\] 
 
\begin{theorem}
Let $X$ be an irreducible $\mathbb{H}$-module. Then $[\mathbb{D}(X)]=\widetilde{\mathbb{D}}([X])$.
\end{theorem}

\begin{proof}
The proof is similar to the one of \cite[Theorems 1 and 2]{Ka0} for Hecke algebras. By Proposition \ref{prop ind res resol}, $\widetilde{\mathbb{D}}([X])=[\ker \pi_0]$, where $\pi_0$ is defined as above or as in Section \ref{ss ind res resol}. By Lemma \ref{lem des ker} and Lemma \ref{lem im D action}, we have $\ker \pi_0 \cong \mathbb{D}(X)$. Hence we obtain the statement.
\end{proof}

\section{Extensions of discrete series} \label{s ext ds}

In this section, we assume $R$ spans $V$. Let $n=\dim V=|\Pi|$. This assumption is more convenient for the discussion of discrete series and it should not be hard to formulate results without this assumption. 

This section is to compute extensions of discrete series by using the resolution (\ref{prop ind res resol}) and the duality of Theorem \ref{thm poin dua}.

\subsection{Extensions of discrete series}

We first state an algebraic definition of a tempered module and a discrete series, following \cite{Ev} and \cite{KR}. Since $V$ admits a natural real form, we can talk about the real part of a $S(V)$-weight of an $\mathbb{H}$-module.
\begin{definition} \label{def temp ds}
An $\mathbb{H}$-module $X$ is said to be a {\it tempered module} if the real part of any $S(V)$-weight $\gamma$ of $X$ has the form:
\begin{align} \label{eqn temp}
   \mathrm{Re} \gamma = \sum_{\alpha \in \Pi} a_{\alpha} \alpha^{\vee}, \quad a_{\alpha} \leq 0 .
\end{align}
Equivalently, $X$ is a tempered module if and only if $\langle \omega_{\alpha^{\vee}},\mathrm{Re}\gamma \rangle \leq 0$ for all $\alpha \in \Pi$ and for all weight $\gamma$ of $X$. Here $\omega_{\alpha^{\vee}}$ is the fundamental weight corresponding to $\alpha^{\vee}$ i.e. $\langle \omega_{\alpha^{\vee}}, \beta^{\vee} \rangle =0$ for $\beta \in \Pi \setminus \left\{ \alpha \right\}$ and $\langle \omega_{\alpha^{\vee}}, \alpha^{\vee} \rangle =1$

An $\mathbb{H}$-module is said to be a {\it discrete series} if all the inequalities in (\ref{eqn temp}) are strict. In particular, any discrete series is tempered.
\end{definition}

\begin{remark}
For the real-valued and equal parameters $k$, discrete series and tempered modules naturally come from the corresponding notion of $p$-adic groups and affine Hecke algebras. They play important roles for the Langlands classification of simple modules, which is also valid for arbitrary parameters. In other words, tempered modules form the basic building blocks for simple modules of graded Hecke algebras for arbitrary parameters (see  \cite[Section 2]{KR} for the details). However, the classification for tempered modules for non-real-valued and unequal parameters seems to be less studied.
\end{remark}

An important ingredient for computing extensions between a tempered module and a discrete series is the following vanishing result.
\begin{lemma}\label{lem proper vanishing}
Let $X_1$ be an irreducible tempered module and let $X_2$ be an irreducible discrete series. Let $J$ be a proper subset of $\Pi$. Then 
\[  \mathrm{Ext}_{\mathbb{H}}^i(\mathrm{Ind}_{\mathbb{H}_J}^{\mathbb{H}} \mathrm{Res}_{\mathbb{H}_J}X_1, \mathbb{D}(X_2)) =0 \] 
for all $i$. 
\end{lemma}

\begin{proof}
Pick $\beta \in \Pi \setminus J$. Set $\omega=\omega_{\beta^{\vee}}$. Then for any weight $\gamma_1$ of $X_1$, $\langle \omega, \mathrm{Re}\gamma_1 \rangle \leq 0$ and so $\mathrm{Re}\gamma_1$ is of the form $\sum_{\alpha \in \Pi \setminus \left\{ \beta \right\}} a_{\alpha}\alpha^{\vee}+a_{\beta}\omega^{\vee}$, where $a_{\alpha} \in \mathbb{R}$, $a_{\beta} \leq 0$ and $\omega^{\vee}$ is the fundamental coweight corresponding to $\beta$. For any weight $\gamma_2$ of $\mathbb{D}(X_2)$, $\langle  \omega, \mathrm{Re}\gamma_2 \rangle >0$ and so $\mathrm{Re}\gamma_2$ is of the form $\sum_{\alpha \in \Pi \setminus \left\{ \beta \right\}} b_{\alpha}\alpha^{\vee}+b_{\beta}\omega^{\vee}$, where $b_{\alpha} \in \mathbb{R}$ and $b_{\beta}>0$. Hence $\gamma_1$ and $\gamma_2$ cannot be $W_J$ conjugate to each other. Thus $\mathrm{Res}_{\mathbb{H}_J}X_1$ and $\mathrm{Res}_{\mathbb{H}_J}\mathbb{D}(X_2)$ have distinct $\mathbb{H}_J$-central characters. Hence, by Proposition \ref{prop char ext 0}, for all $i$,
\[  \mathrm{Ext}_{\mathbb{H}_J}^i(\mathrm{Res}_{\mathbb{H}_J}X_1, \mathrm{Res}_{\mathbb{H}_J}\mathbb{D}(X_2))=0. \]
Then by Frobenius reciprocity, we obtain the statement.
\end{proof}

\begin{theorem} \label{thm ds ext}
 Let $\mathbb{H}$ be the graded affine Hecke algebra associated to a root datum $( R, V, R^{\vee}, V^{\vee}, \Pi)$ and a parameter function $k: \Pi \rightarrow \mathbb{C}$ (Definition \ref{def graded affine}). Assume $R$ spans $V$. Let $X_1$ be an irreducible tempered module and let $X_2$ be an irreducible discrete series. Then
\[     \mathrm{Ext}^i_{\mathbb{H}}(X_1, X_2) \cong \left\{ \begin{array}{cc}
                                                       \mathbb{C} & \quad \mbox{ if $X_1 \cong X_2$ and $i=0$ } \\
																											 0       & \quad \mbox{ otherwise . }
																					\end{array}  \right.
\]

\end{theorem}

\begin{proof}
By using (\ref{eqn ind-res resol}), Lemma \ref{lem im D action}, Lemma \ref{lem des ker} and Proposition \ref{prop ind res resol}, we have a resolution for $X_1$ of the following form:
\[  0 \rightarrow \mathbb{D}(X_1) \rightarrow \bigoplus_{|J|=0} \bigoplus_{J} \mathrm{Ind}_{\mathbb{H}_J}^{\mathbb{H}}\mathrm{Res}_{\mathbb{H}_J}X_1 \rightarrow \ldots \rightarrow \bigoplus_{|J|=n-1} \bigoplus_{J} \mathrm{Ind}_{\mathbb{H}_J}^{\mathbb{H}}\mathrm{Res}_{\mathbb{H}_J}X_1 \rightarrow X_1 \rightarrow 0
\]
Taking an injective resolution for $\mathbb{D}(X_2)$,
\[ 0 \rightarrow \mathbb{D}(X_2) \rightarrow I_0 \rightarrow I_1 \rightarrow \ldots, \]
 we obtain a double complex $C_{p,q}$ such that 
\[  C_{p,q}= \left\{ \begin{array}{ll}
                                                      \mathrm{Hom}_{\mathbb{H}}\left(\bigoplus_{|J|=n-1-p}  \bigoplus_{J} \mathrm{Ind}_{\mathbb{H}_J}^{\mathbb{H}}\mathrm{Res}_{\mathbb{H}_J}X_1, I_q\right)  & \quad \mbox{ if $0 \leq p \leq n-1$ and $0 \leq q$ } \\
																					\mathrm{Hom}_{\mathbb{H}}(\mathbb{D}(X_1), I_q) & \quad \mbox{ if $p=n$ and $0 \leq q$} \\
																											 0       & \quad \mbox{ otherwise  }
																					\end{array}  \right. ,
\]
By considering two spectral sequences associated to the double complex (see e.g. \cite[Chapter 5.6]{We}), we obtain a spectral sequence with the $E_1$ terms
\[  E_1^{p,q}= \left\{ \begin{array}{ll}
                                                      \mathrm{Ext}_{\mathbb{H}}^{q}\left(\bigoplus_{|J|=n-1-p}  \bigoplus_{J} \mathrm{Ind}_{\mathbb{H}_J}^{\mathbb{H}}\mathrm{Res}_{\mathbb{H}_J}X_1, \mathbb{D}(X_2)\right)  & \quad \mbox{ if $0 \leq p \leq n-1$ and $0 \leq q$ } \\
																					\mathrm{Ext}_{\mathbb{H}}^q(\mathbb{D}(X_1), \mathbb{D}(X_2)) & \quad \mbox{ if $p=n$ and $0 \leq q$} \\
																											 0       & \quad \mbox{ otherwise  }
																					\end{array}  \right. ,
\]
and
$ E_1^{p,q}\Rightarrow \mathrm{Ext}^{p+q}_{\mathbb{H}}(X_1, \mathbb{D}(X_2))$. 

Now by Lemma \ref{lem proper vanishing}, we have $\mathrm{Ext}_{\mathbb{H}}^{i}(X_1, \mathbb{D}(X_2))=0$ for $i \leq n-1$. By Theorem \ref{thm poin dua}, Proposition \ref{prop equiv ext group} and Definition \ref{def dualizing module}, we have $\mathrm{Ext}_{\mathbb{H}}^i(X_1, X_2)=0$ for all $i \geq 1$. For $i=0$, we have $\mathrm{Ext}^i_{\mathbb{H}}(X_1, X_2)=\mathrm{Hom}_{\mathbb{H}}(X_1, X_2)$ and then the statement follows from Schur's Lemma.
\end{proof}

\section{Euler-Poincar\'e pairing} \label{s epp}
\subsection{Euler-Poincar\'e pairing}
We keep using the notation of a graded affine Hecke algebra in Definition \ref{def graded affine}. (In this section, $R$ does not necessarily span $V$.)
Define the Euler-Poincar\'e pairing for finite-dimensional $\mathbb{H}$-modules $X$ and $Y$ as:
\[  \mathrm{EP}_{\mathbb{H}}(X, Y) = \sum_{i} (-1)^i \dim \Ext_{\mathbb{H}}^i(X,Y).
\]
By Corollary \ref{cor finite dim ext}, $\mathrm{EP}_{\mathbb{H}}$ is well-defined. This pairing can be realized as an inner product on a certain elliptic space for $\mathbb{H}$-modules analogous to the one in $p$-adic reductive groups in the sense of Schneider-Stuhler \cite{SS}. 

The elliptic pairing $\langle , \rangle_W^{\mathrm{ellip}, V}$ for $W$-representations $U$ and $U'$ is defined as
\[  \langle U, U' \rangle_W^{\mathrm{ellip}} =\frac{1}{|W|} \sum_{w \in W} \tr_U(w)\overline{\tr_{U'}(w)}\mathrm{det}_V(1-w) .
\]

\begin{proposition} \label{thm euler poincare}
For any finite-dimensional $\mathbb{H}$-modules $X$ and $Y$, 
\[  \mathrm{EP}_{\mathbb{H}}(X, Y) = \langle \Res_W(X), \Res_W(Y) \rangle_W^{\mathrm{ellip}} .\]
In particular, the Euler-Poincar\'e pairing depends only on the $W$-module structure of $X$ and $Y$.
\end{proposition}
\begin{proof}
\begin{eqnarray*}
& &\mathrm{EP}_{\mathbb{H}}(X, Y)\\
 &=& \sum_{i} (-1)^i \dim \Ext_{\mathbb{H}}^i(X,Y)  \\
                               &=& \sum_i (-1)^i (\ker d_i^{\vee}-\im d_{i-1}^{\vee}) \quad \mbox{(by Proposition \ref{prop complex ext gp})}  \\                           
															&=& \sum_{i} (-1)^i \dim \Hom_{\mathbb{C}[W]} (\Res_W(X)\otimes \wedge^iV, \Res_W(Y)) \quad \mbox{(by Euler-Poincar\'e principle)} \\
															&=&  \sum_{w \in W} \mathrm{tr}_{\Res_WX}(w) \overline{\mathrm{tr}_{\Res_WY} (w) }  \mathrm{tr}_{\wedge^{\pm} V}(w)    \\
															&=& \langle \Res_W(X), \Res_W(Y) \rangle_{W}^{\mathrm{ellip}}  
\end{eqnarray*}
Here $\wedge^{\pm} V=\bigoplus_{i \in \mathbb{Z}} (-1)^i \wedge^i V$ as a virtual representation. The last equality follows from $\mathrm{tr}_{\wedge^i V}(w)=\det(1-w)$ and definitions.
\end{proof}

\subsection{Applications}
We give two applications of the Euler-Poincar\'e pairing in this section.

An element $w \in W$ is said to be {\it elliptic} if $\det_{V}(1-w) \neq 0$. A conjugacy class of $W$ is said to be elliptic if any element in the conjugacy class is elliptic. The first application is to give an upper bound for the number irreducible discrete series. The proof is very similar to the one for affine Hecke algebra by Opdam-Solleveld \cite[Proposition 3.9]{OS} and we omit the details.

\begin{corollary} \label{cor dimension bound}
Let $\mathbb{H}$ be the graded affine Hecke algebra associated to a root datum and an arbitrary parameter function. The number of irreducible $\mathbb{H}$-discrete series is less than or equal to the number of elliptic conjugacy classes of $W$. In particular, there is only finite number of irreducible $\mathbb{H}$-discrete series (up to isomorphism).

\end{corollary}

In particular, for type $H_3$, using the analysis in \cite[Section 7]{Kri},  there are four discrete series of a regular central character. Corollary \ref{cor dimension bound} then implies those four discrete series account for all the discrete series. 




The second application concerns the duals of discrete series. For real parameter functions, it is even known that those discrete series are even $*$-unitary (by some analytic results in affine Hecke algebras, see \cite[Theorem 7.2]{So}). 

\begin{corollary}
Let $\mathbb{H}$ be the graded affine Hecke algebra associated to a root datum and an arbitrary parameter function. Let $X$ be an irreducible $\mathbb{H}$-discrete series. Then 
\begin{enumerate}
\item[(1)] $X$, $X^{\bullet}$, $X^*$ and $\theta(X)$ are isomorphic,
\item[(2)] Let $W\gamma$ be the central character of $X$. Then $W\theta(\gamma)=W\gamma$.
\end{enumerate}

\end{corollary}

\begin{proof}
Since $X$ and $X^{*}$ have the same $W$-module structure, by Proposition \ref{thm euler poincare} we have
\[\mathrm{EP}_{\mathbb{H}}(X, X^*) =\mathrm{EP}_{\mathbb{H}}(X,X)=1. \]
The second equality follows from Theorem \ref{thm ds ext}. Since $X^*$ is also a discrete series (by Lemma \ref{lem two dual theta}), Theorem \ref{thm ds ext} forces $X^* \cong X$. The assertion for $X^{\bullet}$ and $\theta(X)$ in (1) can be proven similarly. For (2), the central character of $\theta(X)$ is $W\theta(\gamma)$. Then (2) follows from $X \cong \theta(X)$ in (1). 
\end{proof}

\begin{remark}
One more application is to give a closed $W$-character formula for discrete series in non-crystallographic types and complex parameter cases as the one in \cite{CT}. Our results of Theorem \ref{thm ds ext} and Proposition \ref{thm euler poincare} replace \cite[Lemma 3.4]{CT} and then the argument in the proof of \cite[Theorem 3.5]{CT} applies.

\end{remark}


\begin{thebibliography}{AFMO}

\bibitem{ALTV} J. Adams, M. van Leeuwen, P. Trapa and D. Vogan, {\it Unitary representations of real reductive groups}, arXiv:1212.2192.

\bibitem{AP} J. Adler and D. Prasad, {\it Extensions of representations of $p$-adic groups}, Nagoya Math. J. {\bf 208} (2012), 171-199.

\bibitem{Au} A. Aubert, Dualit\'e dans le groupe de Grothendieck de la categorie des repr\'esentations lisses de longueur finie d'un groupe reductif p-adique, Transactions of the American Mathematical Society {\bf 347} (1995), no. 6, 2179-2189.


\bibitem{BC} D. Barbasch and D. Ciubotaru, {\it Hermitian forms for affine Hecke algebras}, arXiv:1312.3316v1 [math.RT].

\bibitem{BC2} D. Barbasch and D. Ciubotaru, {\it Star operations for affine Hecke algebras}, arXiv:1504.04361 [math.RT].

\bibitem{BCT} D. Barbasch, D. Ciubotaru and P. Trapa, {\it Dirac cohomology for graded affine Hecke algebras}, Acta. Math. {\bf 209} (2) (2012), 197-227.

\bibitem{BM} D. Barbasch and A. Moy, {\it A unitarity criterion for p-adic groups}, Invent. Math. {\bf 98} (1989), 19-37.

\bibitem{BM1} D. Barbasch and A. Moy, {\it Reduction to real infinitesimal character in affine Hecke algebras}, Jour. of Amer. Math. Soc. (1993), 611-635.



\bibitem{BW} A. Borel and N. Wallach, {\it Continuous cohomology, discrete subgroups, and representations of reductive groups}, {\bf 94} (1980), Princeton University Press (Princeton, NJ).





\bibitem{Ch} K. Y. Chan, {\it Extensions of graded affine Hecke algebra modules}, PhD thesis, University of Utah, 2014.

\bibitem{Ch2} K. Y. Chan, {\it First extensions and filtrations of standard modules for graded Hecke algebras}, preprint, arXiv:1510.05410 (2015).







\bibitem{CK} D. Ciubotaru and S. Kato, {\it Tempered modules in exotic Deligne-Langlands correspondence}, Adv. Math. {\bf 226} (2011) no. 2, 1538-1590.

\bibitem{CT} D. Ciubotaru and P. Trapa, {\it Characters of Springer representations on elliptic classes}, Duke Math. J. (2) {\bf 162} (2013), 201-223. 

\bibitem{Dr} V.G. Drinfeld, Degenerate affine Hecke algebras and Yangians, Funct. Anal. Appl. {\bf 20} (1986) 58-60.



\bibitem{Ev} S. Evens, {\it The Langlands classification for graded Hecke algebras}, Proc. Amer. Math. Soc. 124 (1996), no. 4, 1285-1290.

\bibitem{EM} S. Evens and I. Mirkovi\'c, {\it Fourier transform and the Iwahori-Matsumoto involution}, Duke Math. J. {\bf 86} (1997), no. 3, 435-464.




 \bibitem{HP} J. S. Huang and P. Pand$\check{\mathrm{z}}$i\'{c}, {\it Dirac operators in representation theory}, Mathematics: Theory \& applications (2006), Birkh\"{a}user.

\bibitem{Hu} J. Humphreys, {\it Reflection groups and Coxeter groups}, Cambridge Univ. Press, 1990.


\bibitem{Ka0} S.-I. Kato, {\it Duality for representations of a Hecke algebra},  Proceedings of Amer. Math. Soc. {\bf 119} (3), 941-946.



\bibitem{Ka} S. Kato, {\it An exotic Deligne-Langlands correspondence for symplectic groups}, Duke Math. J. 148 (2009), 305-371.

\bibitem{Ka2} S. Kato, {\it A homological study of Green polynomials}, Ann. Sci. l'\'Ecole Norm. Sup. {\bf 48}, no. 5 (2015), 1035-1074.


\bibitem{Kn} A. Knapp, {\it Lie groups, Lie algebras and cohomology}, Mathematical Notes 34, Princetion University Pres, 1988.

\bibitem{KL} D. Kazhdan and G. Lusztig, {\it Proof for the Deligne-Langlands conjecture for Hecke algebras}, {\it Invent. Math.},  {\bf 87} (1987), 153-215.

\bibitem{Kri} C. Kriloff, {\it Some interesting nonspherical tempered representations of graded Hecke algebras}, Trans. Amer. Math. Soc. {\bf 351} (1999), 4411-4428.

\bibitem{KR} C. Kriloff, A. Ram, {\it Representations of graded Hecke algebras},  Represent. Theory {\bf 6} (2002), 31-69.



\bibitem{Lu} G. Lusztig, {\it Affine Hecke algebras and their graded versions}, J. Amer. Math. Soc. {\bf 2} (1989), 599-635.

\bibitem{Ma} S. MacLane, {\it Homology}, Grundlehren der mathematischen Wissenschaften {\bf 114}, Springer-Verlag, 1975.

\bibitem{Me} R. Meyer, {\it  Homological algebra for Schwartz algebras of reductive p-adic groups}, Noncommutative geometry and number theory, Aspects of Mathematics E37 (2006), Vieweg Verlag, Wiesbaden, 263-300.

\bibitem{Na} M. Nazarov, {\it Young's symmetrizers for projective representations of the symmetric group}, Adv. Math. {\bf 127} (1997), no. 2, 190-257.


\bibitem{Re} M. Reeder, {\it Euler-Poincar\'e pairings and elliptic representations of Weyl groups and $p$-adic groups}, Compositio Math. {\bf 129} (2001), 149-181.



\bibitem{Op0} E. M. Opdam, {\it Harmonic analysis for certain representations of graded Hecke algebras}, Acta. Math. {\bf 175} (1995), no. 1, 75-121. 

\bibitem{OS} E. M. Opdam and M. Solleveld, {\it Homological algebra for affine Hecke algebras}, Adv. in Math. {\bf 220} (2009), 1549-1601.

\bibitem{OS1} E. M. Opdam and M. Solleveld, {\it Discrete series characters for affine Hecke algebras and their formal degrees}, Acta Mathematica. {\bf 205} (2010), 105-187.

\bibitem{OS2} E. M. Opdam and M. Solleveld, {\it Extensions of tempered representations}, Geometric And Functional Analysis {\bf 23} (2013), 664-714.

\bibitem{Or} S. Orlik, {\it On extensions of generalized Steinberg representations}, J. Algebra {\bf 293} (2005), no. 2, 611-630.

\bibitem{Pr} D. Prasad, {\it Ext-analogues of branching laws}, preprint arXiv:1306.2729v1. 

\bibitem{SS}   P. Schneider and U. Stuhler, {\it Representation theory and sheaves on the Bruhat-Tits building}, Publ. Math. Inst. Hautes Etudes Sci. {\bf 85} (1997), 97-191.

\bibitem{SW} A. Shepler and S. Witherspoon, {\it Hochschild cohomology and graded Hecke algebras}, Trans. Amer. Math. Soc. 360 (2008), no. 8, 3975-4005. 

\bibitem{Sl} K. Slooten, {\it Generalized Springer correspondence and Green function for type B/C graded Hecke algebras}, Adv. in Math. {\bf 203} (2006), 34-108.

\bibitem{So} M. Solleveld, {\it Parabolically induced representations of graded Hecke algebras},  Algebras and Representation Theory {\bf 152} (2012), 233-271. 

\bibitem{Vi}  M. F. Vign\'eras, {\it Extensions between irreducible representations of p-adic GL(n)}, Pacific J. of Maths, {\bf 181} (3) (1997), 349-357.

\bibitem{We} C. Weibel, {\it An introduction to homological algebra}, Cambridge Studies in Advanced Mathematics 38, Cambridge University Press (1994).

\bibitem{Ye} W. L. Yee, {\it Signatures of invariant Hermitian forms on irreducible highest-weight modules}, Duke Math. J. {\bf 142} (2008), no. 1, 165-196.














\end{thebibliography}
\end{document}